\documentclass[12pt]{amsart}
\usepackage{amsfonts,amssymb,eucal}

\usepackage{color}

\usepackage{amsthm} 
 \usepackage{amsmath} 
 \usepackage{latexsym}
\usepackage{bbold,mathbbol,bbm}
\numberwithin{equation}{section}

\newcommand{\tu}{\ensuremath{\tau}}
\newcommand{\sg}{\ensuremath{\sigma}}

\newcommand{\mb}[1]{{\mbox{\boldmath{$#1$}}}}
\newcommand{\mc}[1]{{\mathcal{#1}}}
\newcommand{\got}[1]{{\mathfrak{#1}}}
\newcommand{\db}[1]{{\mathbb{#1}}}
\newcommand{\pa}{\partial}

\newcommand{\R}{\ensuremath{\mathbb{R}}}
\newcommand{\C}{\ensuremath{\mathbb{C}}}
\newcommand{\N}{\ensuremath{\mathbb{N}}}

\newtheorem{Remark}{Remark}

\newtheorem{Theorem}{Theorem}
\newtheorem{Proposition}{Proposition}
\newtheorem{lemma}{Lemma}

\theoremstyle{definition} 
\def\ii{\operatorname{i}}
 \newcommand{\SL}{\ensuremath{{\mbox{\rm{SL}}(2,\R)}}}
  
\newcommand{\Ka}{K\"ahler}
\newcommand{\mr}[1]{{\mathrm{#1}}}

\newcommand{\dd}{\operatorname{d}}
\newcommand{\grad}{\operatorname{grad}}

\newtheorem{deff}{Definition}
\renewcommand{\Re}{\operatorname{Re}}
\renewcommand{\Im}{\operatorname{Im}}
\allowdisplaybreaks 

\begin{document}
\title{Hamiltonian systems on almost cosymplectic manifolds}
\author{Stefan  Berceanu}
\address[Stefan  Berceanu]{National
 Institute for Physics and Nuclear Engineering\\
         Department of Theoretical Physics\\
         PO BOX MG-6, Bucharest-Magurele, Romania}
         \email{Berceanu@theory.nipne.ro}
 
\begin{abstract}
 We determine the   Hamiltonian vector field on an odd dimensional
 manifold endowed with almost cosymplectic structure. This is a generalization
 of the corresponding Hamiltonian vector field on manifolds with  almost transitive
 contact structures,  which extends  the contact Hamiltonian
 systems. Applications are presented to the equations of motion on   a particular five-dimensional
manifold,  the extended Siegel-Jacobi upper-half plane  $\tilde{\mathcal{X}}^J_1$.
The   $\tilde{\mathcal{X}}^J_1$ manifold  is endowed with a
generalized transitive almost cosymplectic structure,
 an almost cosymplectic structure,  more general  than 
 transitive almost  contact structure and  cosymplectic
  structure. The equations of motion on  $\tilde{\mathcal{X}}^J_1$  extend the Riccati equations of motion
  on the four-dimensional  Siegel-Jacobi manifold  $\mathcal{X}^J_1$ attached to a linear
  Hamiltonian in the generators of the  real Jacobi
  group $G^J_1(\mathbb{R})$.
\end{abstract}

\subjclass{37J55; 53D15; 53D22}
\keywords{Jacobi group, invariant metric,  Siegel--Jacobi disk, Siegel--Jacobi upper
  half-plane,  extended
  Siegel--Jacobi upper half-plane, almost cosymplectic manifold,
  cosymplectic manifold, contact Hamiltonian system, equations of motion}
\maketitle
\today  
\tableofcontents
\section{Introduction}
The semi-direct product  $G^J_n(\R)$ of ${\rm
  Sp}(n,\R)$ with   the $(2n+1)$-dimensional Heisenberg group
$\rm{H}_n(\R)$ is called real Jacobi group of degree $n$, while the
isomorph  group $ \mr{H}_n \ltimes 
 \mr{ Sp}(n,\R)_{\C}$ is denoted  $G^J_n$. Both Jacobi  groups
$G^J_n(\R)$ and $G^J_n$
are intensively studied in Mathematics,  Mathematical Physics and
Theoretical Physics  \cite{jac1,SB15,SB19,BERC08B,gem,bs,ez},
\cite{yang}-\cite{Y10}.

The $G^J_n(\R)$-homogeneous space 
$\mc{X}^J_n:=\frac{G^J_n(\R)}{\rm{U}(n)\times \R}$
is called Siegel-Jacobi  upper half space and  $\mc{X}^J_n\approx \mc{X}_n\times
\R^{2n}$, where  $\mc{X}_n$ denotes the Siegel upper half space
\cite{SB15,SB19,yang,Y08}.   Similarly, the Siegel-Jacobi ball  is
defined as 
 $\mc{D}^J_n:=  \frac{G^J_n}{\rm{U}(n)\times \R}\approx \mc{D}_n\times\C^n$ \cite{sbj},  where  $\mc{D}_n$ denotes the Siegel (open)  ball  of degree
$n$ \cite{helg}.

Systems of coherent states (CS) based on the Siegel-Jacobi ball  have applications
in  quantum mechanics, geometric quantization,
dequantization, quantum optics, squeezed states, quantum
teleportation,  quantum tomography,  nuclear structure,  signal  processing, 
Vlasov kinetic equation, see references in   \cite{SB20,GAB,SB19}. 

In order to construct invariant metric on  homogeneous spaces
attached to the real Jacobi group of degree one, in \cite{SB19} we
 gave up   CS technique  inspired by Berezin \cite{ber74}-\cite{berezin}
and  we applied   moving frame  method   \cite{ev},  initiated by Cartan
\cite{cart4,cart5}. In \cite{SB19} we got invariant metrics on   $\mc{X}^J_1$ and 
 $\tilde{\mc{X}}^J_1$, while in \cite{SB20N} the results were 
 generalized     to extended Siegel-Jacobi upper half space
   $\tilde{\mc{X}}^J_n\approx \mc{X}^J_n\times\R$.

   It  raises the question: {\it how can we distinguish
 between the different invariant metrics obtained in} \cite{SB19} on  
$\mc{X}^J_1$ and $\tilde{\mc{X}}^J_1$?  
In the present paper we  give an answer to this problem, 
underlining 
 the differences between equations of motion on   $\mc{X}^J_1$ and 
 $\tilde{\mc{X}}^J_1$.

We developed  a technique to calculate equations of
classical and quantum
motion on a homogenous manifold $M=G/H$ attached to a linear 
Hamiltonian  in the generators of the group $G$, simultaneously with the
determination of the Berry phase \cite{sb6}-\cite{FC}.  The method works on the  so called
CS manifolds, i.e.  {\it  \Ka~ manifolds} for which the
generators of group $G$ admits  a realisation 
 as first order holomorphic differential operators with
polynomial coefficients \cite{sb6,SBAG01,last}.

In this paper we are interested in a similar problem: 
{\it find the equations of motion on  odd dimensional manifolds}
$M_{2n+1}$ with   {\it almost cosymplectic structure} in the
meaning of \cite{paul}. For an almost
cosymplectic manifold $(M_{2n+1},\theta, \Omega)$, where $\theta $ ($\Omega$) is a
one (respectively,  two)-form,   we determine the equations
of motion attached to a Hamiltonian function $H$. 

We determined the invariant metric to  the action of the Jacobi group
$G^J_1(\R)$ on $\tilde{\mc{X}}^J_1$ \cite{SB19}. 
We pointed  out that the extended Siegel--Jacobi
manifold does not admit a Sasaki structure in the parametrization used
in \cite{SB19}.
In the present paper we introduce 
on  $\tilde{\mc{X}}^J_1$ an   almost cosymplectic structure in the sense
of Libermann \cite{paul}, and moreover, we consider the particular
 case when 
$\dd \Omega =0$. 
 Such a manifold is  called {\it  generalized transitive almost
   cosymplectic  space} (GTACOS). The place of the GTACOS structure
in the set of   geometric structures with which a manifold of odd dimension can be endowed is underlined in
the Table in the Appendix: GTACOS  is an almost cosymplectic structure more general than
the transive almost contact structure and contact structure.

 We also endow  $\tilde{\mc{X}}^J_1$ with a contact structure in the sens of \cite{MLEO}.


The paper is laid out as follows. In Section    \ref{PR} are collected
several previous results extracted from \cite{SB20,SB19,SB21}
used in 
 this  paper.
Proposition \ref{BIGTH} presents the invariant metrics on five
$G^J_1(\R)$-homogeneous manifolds. Proposition \ref{PRFC} expresses
the \Ka~
 two-form and the invariant metric on $\mc{D}^J_1$ and  
$\mc{X}^J_1$ in several sets of variables. Proposition \ref{PROP5} recalls the invariant metric on
$\tilde{\mc{X}^J_1}$ in the S-variables $(x,y,p,q,\kappa)$ \cite{bs,ez}. In Section
\ref{RZ} are calculated the Hamiltonian vector field $X_H$  and $\grad H$ associated
with  the Hamiltonian function $H$ and the corresponding equations of
motion on almost cosymplectic manifolds in the variables
$(q^i,p_i,\kappa)$ , $i=1,\dots,n$, which  appear in the one-form $\theta$,  while
$\Omega$ is 
expressed in Darboux coordinates (Theorem \ref{main}). The new results on
almost cosymplectic manifolds  are
compared with the corresponding results for transitive almost contact manifolds 
\cite{Albert} (Remark
\ref{RR1}). Section \ref{ACHS} particularizes the results of Section
\ref{RZ} to the extended Siegel-Jacobi upper half-plane    $\tilde{\mc{X}}^J_1$    organized as
 GTACOS manifold.  Proposition \ref{P16} displays  the
equations of motion  on  $\tilde{\mc{X}}^J_1$  in the variables
$(x,y,p,q,\kappa)$ organized as a GTACOS manifold for a Hamiltonian function 
$H$ comparatively to the equations of motion on  $\mc{X}^J_1$ in the
variables  $(x,y,p,q)$. Proposition \ref{PPP}
underlines
 that  $\tilde{\mc{X}}^J_1$  does not admit an
 almost contact structure with the invariant metric $g_{\tilde{\mc{X}}^J_1} $
constructed in \cite{SB19}, but  $\tilde{\mc{X}}^J_1$  can be
equipped 
with an almost contact
metric structure  with a metric different of the metric 
$g_{\tilde{\mc{X}}^J_1}$.  We underline 
 that the parametrization introduced in
\cite{god} for Sasaki potential on Sasaki manifolds is not appropriate 
for $\tilde{\mc{X}}^J_1$ (Remark \ref{RR223}).   Proposition
\ref{PRR4} of Section
\ref{CHS} shows  the equations of motion on $\tilde{\mc{X}}^J_1$
organized as contact Hamiltonian system. Proposition \ref{PR6} in
Section \ref{43} 
presents the equations of motion on $\tilde{\mc{X}}^J_1$ organized as
GTACOS manifold generated by a Hamiltonian linear in
generators of the real Jacobi group comparatively with the
corresponding equations on ${\mc{X}}^J_1$. Finally,  the  Appendix
summarizes the definitions of different mathematical concepts used in
the paper. 

The new results of the paper are contained in  Lemma \ref{L1}, the  six  Remarks, 
Propositions  \ref{P16} -- \ref{PR6} and Theorem \ref{main}, which
presents  the equations of motion on odd
dimensional manifolds endowed with almost cosymplectic structure. The Remark \ref{R333} emphasizes
   the differences between the equations of motion
   on $\tilde {\mc {X}} ^ J_1 $ endowed with
almost cosymplectic structure
$ (\tilde{\mc {X}} ^ J_1, \theta, \omega) $ and contact structure
$ (\tilde {\mc {X}} ^ J_1,\eta_0) $, which in principle could be
experimentally highlighted.
Although   we are  not  able to determine effectively the
almost contact structure associated with the concrete realization of the
contact structure  $(\tilde{\mc{X}}^J_1,\eta_0)$, we report  the
partial results obtained at    Proposition \ref{PPP}  b) in order to emphasize the difficulty of the
problem.

\textbf{Notation}
We denote by $\mathbb{R}$, $\mathbb{C}$, $\mathbb{Z}$ and $\mathbb{N}$ 
 the field of real numbers, the field of complex numbers,
the ring of integers,   and the set of non-negative integers, respectively. We denote the imaginary unit
$\sqrt{-1}$ by~$\ii$, the real and imaginary parts of a complex
number $z\in\C$ by $\Re z$ and $\Im z$ respectively, and the complex 
conjugate of $z$ by $\bar{z}$.
 We denote by ${\dd }$ the differential. 
We use Einstein's summation convention, i.e.  repeated indices are
implicitly summed over.  The set of vector fields (1-forms) is denoted
by $\got{D}^1$ (respectively $\got{D}_1$). If we
denote with Roman  capital letteres the Lie  groups, then their
associated Lie algebras are denoted with the corresponding lower-case
letteres. The interior  product $i_X\omega$ (interior multiplication or contraction)
of the differential form $\omega$
with $X\in\got{D}^1$ is denoted $X\lrcorner \omega$. We
denote by $M(n,m,\db{F})$ the set of $n\times m$ matrices with elements
in the field $\db{F}$. If $X\in M(n,m,\db{F})$, then $X^t$ denotes the
transpose of $X$.

\section{Preparation: Invariant metrics on   $G^J_1(\R)$-homogenous spaces}\label{PR}

In \cite[(4.10), (5.15), (5.17)]{SB19} we have introduced 6 invariant one-forms
$\lambda_1,\dots,\lambda_6$  in the S-coordinates $(x,y,\theta,p,q,\kappa)$ \cite{bs} associated
 with the real Jacobi group $G^J_1(\R)$
\begin{subequations}\label{PARO}
  \begin{align}
    \lambda_1 & =\frac{\sqrt{\alpha}}{y}(\cos 2\theta\dd x+ \sin
                2\theta \dd y), \quad \alpha>0,\\
\lambda_2& =\frac{\sqrt{\alpha}}{y}(-\sin 2\theta \dd x+\cos 2\theta \dd y),\\
\lambda_3 & =\sqrt{\beta} (\frac{\dd x}{y} + 2\dd
            \theta), \label{3ll}\quad \beta>0,\\
            \lambda_4&= \sqrt{\gamma}[-y^{-\frac{1}{2}}\sin \theta \dd q +\big(y^{\frac{1}{2}}\cos\theta
 -xy^{-\frac{1}{2}}\sin\theta \big)\dd p],\label{LDP} \quad \gamma>0,\\
    \lambda_5& =\sqrt{\gamma}[y^{-\frac{1}{2}}\cos\theta \dd q
 +\big(y^{\frac{1}{2}}\sin\theta+xy^{-\frac{1}{2}}\cos\theta\big)\dd p],\label{LDQ}\\
    \lambda_6&=\sqrt{\delta}( \dd \kappa -p\dd q +q\dd p),\qquad \delta>0,
  \end{align}
\end{subequations}
and we have expressed the
invariant metric  on several  homogenous spaces associated with
$G^J_1(\R)$    \cite[Theorem 5.7]{SB19}:
\begin{Proposition}\label{BIGTH}
The four-parameter left-invariant metric on the real Jacobi group\\ $G^J_1(\R)$ in the S-coordinates $(x,y,\theta,$ $p,q,\kappa)$ is
\begin{equation}\label{MTRTOT}
  \begin{split}
{\rm d}s^2_{G^J_1(\R)} &=\sum_{i=1}^6\lambda_i^2 =\alpha\frac{{\rm d}
  x^2+{\rm d} y^2}{y^2} +\beta\left(\frac{{\rm d} x}{y}+2{\rm
    d}\theta\right)^2, \\
& + \frac{\gamma}{y}\big({\rm d}
q^2+ S {\rm d} p^2+2x{\rm d} p{\rm d}q\big)+\delta({\rm
  d} \kappa-p{\rm d} q+q{\rm d} p)^2, ~~S:=x^2+y^2.
\end{split}
\end{equation}

Depending on  the parameter values $\alpha, \beta, \gamma, \delta \ge 0$
in the metric \eqref{MTRTOT}, we have invariant metric on the
following $G^J_1\R)$-homogeneous manifolds$:$
\begin{enumerate}\itemsep=0pt
\item[$1)$] the Siegel upper half-plane $\mc{X}_1$ if $\beta,\gamma,\delta =0$,
\item[$2)$] the group $\SL$ if $\gamma,\delta=0$, $\alpha\beta\not= 0$, 
\item[$3)$] the Siegel--Jacobi half-plane $\mc{X}^J_1$ if $\beta, \delta= 0$, 
\item[$4)$] the extended Siegel--Jacobi half-plane $\tilde{\mc{X}}^J_1$ if $\beta=0$,
\item[$5)$] the Jacobi group $G^J_1$ if $\alpha\beta\gamma\delta\not=
  0$.
\end{enumerate}
\end{Proposition}
The action of Jacobi group   on several  homogeneous spaces
associated with it is presented in
 \cite[Lemma 1]{SB20}, \cite[Lemma 5.1]{SB19}.
    We have studied the geometry of the Siegel--Jacobi
upper half-plane $\mc{X}^J_1$ in  \cite{SB20,jac1,FC, SB19,SB21}.

The parts a), b), c) of Proposition \ref{PRFC} below are extracted  from \cite[Proposition 2.1 in]{SB19},  \cite[Proposition 1]{SB20}. Below $(w,z)\in  ( \mc{D}_1,\C)$,
$(v,u)\in (\mc{X}_1,\C)$. and 
the parameters $k$ and $\nu$ come from representation theory of the
Jacobi group: $k$ indexes the positive discrete series of ${\rm
  SU}(1,1)$, $2k\in\N$, while $\nu>0$ indexes the representations of
the Heisenberg group $\mr{H}_1$.  See also \cite{jac1},
\cite{SB15}. By  the two-form of 
Berndt-\Ka~  we mean 
the   two-parameter invariant \Ka~   two-form on  $\mc{X}^J_1$ determined 
in \cite{bern84,bern,cal3,cal}, see also \cite{SB20,SB19}. Part d)  is
taken from \cite[Proposition 3]{SB21}.     Part e) 
was proved in \cite[Proposition 5.4, see (5.21b)]{SB19}
and appears also in \cite[Proposition 7, see (4.18),  (4.23)]{SB21}.
\begin{Proposition} \label{PRFC}

  a) The \Ka~two-form on  $\mc{D}^J_1$, invariant to the action of
  $G^J_0=\rm{SU}(1,1)\ltimes\C$, is 
 \begin{equation}\label{kk1}
  -\ii \omega_{\mc{D}^J_1}
(w,z)\!=\!\frac{2k}{P^2}\dd w\wedge\dd
  \bar{w}+\nu \frac{\mc{A}\wedge\bar{\mc{A}}}{P},~P:=1-|w|^2,~\mc{A}=\mc{A}(w,z):=\dd
  z+\bar{\eta}\dd w.
  \end{equation}

We have the change of variables $FC: (w,z)\rightarrow (w,\eta)$
\begin{gather*}
{\rm FC}\colon \
 z=\eta-w\bar{\eta},\qquad {\rm FC}^{-1}\colon \
 \eta=\frac{z+\bar{z}w}{P},
\end{gather*}
and
\[
  {\rm FC}\colon \ \mc{A}(w,z)\rightarrow \dd \eta -w\dd
  \bar{\eta}.
\]

 b) Using the partial Cayley transform  $\Phi~:(w,z)\rightarrow (v,u)$ and its
inverse 
\begin{subequations}
\begin{align}
\Phi: w & =\frac{v-\ii}{v+\ii},~~z=2\ii
        \frac{u}{v+\ii},~~v,u\in\C,~\Im v>0,\label{210b}\\
\Phi^{-1}: v & =\ii \frac{1+w}{1-w},~~u=\frac{z}{1-w}, ~~w,z\in\C,~ |w|<1,\label{210a}
\end{align}
\end{subequations}
we obtain
\[
\mc{A}\left(\frac{v - \ii}{v+ \ii},\frac{2\ii
      u}{v + \ii}\right)=\frac{2\ii}{v+\ii}\mc{B}(v,u),
  \]
  where
  \begin{equation}\label{BFR2}
  \mc{B}(v,u) := {\rm d} u - r{\rm d} v, ~
  r:=\frac{u-\bar{u}}{v-\bar{v}}.
  \end{equation}
  The Berndt--\Ka 's two-form, invariant to the action of
$G^J(\R)_0=\rm{SL}(2,\R)\ltimes\C$, is 
\begin{equation}
- \ii \omega_{\mc{X}^J_1}(v,u) = -\frac{2k}{(\bar{v} - v)^2} \dd
v\wedge \dd\bar{v}+ \frac{2\nu}{\ii(\bar{v} - v)}\mc{B}\wedge\bar{\mc{B}}. \label{BFR}
\end{equation}
We have the   change of variables ${\rm FC}_1\colon \
 (v,u)\rightarrow (v,\eta)$
 \[
   {\rm FC}_1\colon \ 2\ii u=(v+\ii)\eta-(v-\ii)\bar{\eta}, \qquad
   {\rm FC}^{-1}_1 \colon  \eta=\frac{u\bar{v}-\bar{u}v}{\bar{v}-v} +
   \ii r.
   \]

 c) If we apply the change of coordinates $\mc{D}^J_1\ni
 (v,u)\rightarrow 
 (x,y,p,q)\in\mc{X}^J_1$
 \[
   \C\ni   u:=pv+q,~~p,q\in\R, \quad \C\ni v:=x+\ii y, ~x,y\in\R,~ y>0,
   \]
then
\[
  r=p,\quad
 \mc{B}(v,u)=\dd u -p \dd v,
\]
\begin{equation}\label{BUVpq1}
  \mc{B}(v,u)=\mc{B}(x,y,p,q):=F \dd t = v\dd p+\dd q=(x+\ii y)\dd p
  +\dd q, \quad  F:= \dot{p}v+\dot{q}.
\end{equation}

d) The second partial Cayley transform $\Phi_1: \mc{D}^J_1\rightarrow \mc{X}^J_1$
$$\Phi_1:=FC_1\circ\Phi:  (w,z)\rightarrow
(v=x+\ii y,\eta=q+\ii p)$$
 is  given  by  
\begin{subequations}\label{ULTRAN}
\begin{align}
 \Phi_1: & ~w=\frac{v-\ii}{v+\ii}, \quad  z=2\ii \frac{pv+q}{v+\ii},\label{ULTRAN1}\\
\Phi_1 ^{-1}: &~ v=\ii \frac{1+w}{1-w}, \quad \eta =
\frac{(1+\ii \bar{v})(z-\bar{z})+v(\bar{v}-\ii)(z+\bar{z})}{2\ii
  (\bar{v}-v)}=\frac{z+\bar{z}w}{P}.\label{ULTRAN2}
\end{align}
\end{subequations}
e) The two-parameter   balanced  metric  on the
Siegel--Jacobi upper half-plane $\mc{X}^J_1$, the particular case of
\eqref{MTRTOT}  corresponding to  item 
\emph{3)} in \emph{Proposition \ref{BIGTH}},   associated to the \Ka~
two-form \eqref{BFR}, \eqref{BUVpq1}, is 
\begin{equation}\label{METRS2}
  \dd s^2_{\mc{X}^J_1}(x,y,p,q)  \!=\!
\alpha\frac{\dd x^2\!+\!\dd   y^2}{y^2} +\frac{\gamma}{y}(S\dd p^2+\dd q^2+2x\dd
  p\dd q),
\end{equation}
where $S$ was defined in \eqref{MTRTOT} and
\begin{equation}\label{AK}
  \alpha:=k/2,\quad\gamma:=\nu.
\end{equation}
The metric matrix associated  with \eqref{METRS2} is
\[
g_{{\mc{X}}^J_1}\! = \!\left(\begin{array}{cccc}g_{xx} &0 &0 &0\\
0& g_{yy}& 0& 0 \\
0& 0& g_{pp} & g_{pq} \\0 & 0& g_{qp}& g_{qq}
 \end{array}\right),\!
 \begin{array}{cc}\qquad g_{xx}\!=\frac{\alpha}{y^2}, &
 \!g_{yy}\!=\!\frac{\alpha}{y^2};\\
g_{pq}\!=\!\gamma\frac{x}{y} , &
g_{pp} \!=\!\gamma\frac{S}{y},\quad g_{qq}\!=\!\frac{\gamma}{y}.
\end{array}
\]
\end{Proposition}

We have also obtained invariant metric to the action of the Jacobi group $G^J_1(\R)$ on the
extended Siegel-Jacobi upper half-plane $\tilde{\mc{X}}^J_1$  
\cite[Proposition 5.6,  (5.25), (5.26)]{SB19}, see also
\cite[Theorem 1,  (5.1) (5.5)]{SB21}
\begin{Proposition}\label{PROP5}
The  three-parameter   metric of  the extended  Siegel-Jacobi upper
  half-plane  
  $\tilde{\mc{X}}^J_1$ expressed  in the  S-coordinates
  $(x,y,p,q,\kappa)$,  left-invariant with  respect  to  the action  of the Jacobi group
 $G^J_1(\R)$,  is given  by  item  \emph{4)} in
  \emph{Proposition \ref{BIGTH}} as 
  \begin{equation}\label{linvG}
    \begin{split}
 {\rm d} s^2_{\tilde{\mc{X}}^J_1}(x,y,p,q,\kappa) &
 \!=\!{\rm d} s^2_{\mc{X}^J_1}(x,y,p,q)+\lambda^2_6(p,q,\kappa)\\
 &\!=\!\frac{\alpha}{y^2}\big({\rm d} x^2+{\rm d}
 y^2\big)+\left(\frac{\gamma}{y}S+\delta q^2\right){\rm d} p^2+
 \left(\frac{\gamma}{y}+\delta p^2\right){\rm d} q^2 +\delta {\rm d} \kappa^2\\
& \!+ 2\left(\gamma\frac{x}{y}-\delta pq\right){\rm d} p{\rm d} q +2\delta (q{\rm d} p{\rm d}
 \kappa-p{\rm d} q {\rm d} \kappa).
    \end{split}
  \end{equation}
   The metric matrix associated  to the  metric \eqref{linvG} is
\begin{equation}\label{begGG}
g_{\tilde{\mc{X}}^J_1}\! = \!\left(\begin{array}{ccccc}g_{xx} &0 &0 &0&0\\
0& g_{yy}& 0& 0 & 0\\
0& 0& g_{pp} & g_{pq} & g_{p\kappa}\\0 & 0& g_{qp}& g_{qq}
 &g_{q\kappa}\\
0& 0& g_{\kappa p}& g_{\kappa q} & g_{\kappa\kappa}
 \end{array}\right),\!
 \begin{array}{cc}g_{xx}\!=\frac{\alpha}{y^2}, &
  \!g_{yy}\!=\!\frac{\alpha}{y^2},\\
 g_{pq}\!=\!\gamma\frac{x}{y}-\delta p q , &
~g_{p\kappa}\!=\!\delta q, g_{q\kappa}\!=\!-\delta p, \\
   g_{pp} \!=\!\gamma\frac{S}{y}+\delta q^2,&
 g_{qq}\!=\!\frac{\gamma}{y}+\delta p^2,   g_{\kappa\kappa}\!=\!\delta.
\end{array}
\end{equation}
The extended Siegel--Jacobi upper half-plane $\tilde{\mc{X}}^J_1$ does
not admit an almost contact structure $(\Phi,\xi,\eta)$ in the sense
of \emph{Definition \ref{D9}} with a contact form $\eta=\lambda_6$ and Reeb vector $\xi= \operatorname{Ker}(\eta)$.
 \end{Proposition}

\section{Equations of motion on almost cosymplectic manifolds}\label{RZ}
In the present paper we adopt the terminology  of Libermann
\cite{paul} recalled
at {\bf ACOS} in Appendix. We
consider the almost cosymplectic manifold $(M_{2n+1},\theta, \Omega)$
such that condition \eqref{thtOM} is satisfied 
 for the real  one-form $\theta$ 
\begin{equation}\label{t1}
\theta=a_i\dd q^i+b_i \dd p_i+c\kappa, \quad a_i,b_i\in \R, ~ i=1,\dots,n, \quad c\neq 0,
\end{equation}
while for the  two-form $\Omega$ we use the Darboux 
parametrization given
in \eqref{tac2}.

We determine the  Hamiltonian vector field $X_H$  verifying  a
relation similar to \eqref{HAM} for the contact manifold $(M,\eta)$
for the 
 one-form $\theta$
\eqref{t1} and two-form  $\Omega$ \eqref{tac2}:
\begin{Theorem}\label{main}Let  $(M_{2n+1},\theta, \Omega)$
be an   almost cosymplectic manifold and let $H\in C^{\infty}(M)$ be a 
smooth Hamiltonian. 
  Then  the coefficients 
   that  define  the Hamiltonian  vector field $X_H$ attached the
   Hamiltonian  $H$
   \[
     X_H=A_i\frac{\pa}{\pa q^i}+B_i\frac{\pa }{p_i}+C\frac{\pa}{\pa \kappa},
   \]
are solution of the equation
\[
  \flat(X_H)=\dd H-(R\lrcorner H+H)\theta,
  \]
where the Reeb vector which is the solution of the equation \eqref{reeb}  is 
\begin{equation}\label{320}
R=\frac{1}{c}\frac{\pa }{\pa \kappa},
\end{equation}
and the isomorphism $\flat$ is defined in \eqref{bemol}.

It is  obtained  
\begin{equation}\label{ABC}
 A_i =\frac{\pa H}{\pa p^i}-b_iR(H),~
 B_i = -\frac{\pa H}{\pa q_i}+a_iR(H),~
 C  =\frac{1}{c}(-a_i\frac{\pa H}{\pa p^i}+b_i\frac{\pa H}{\pa q^i}-H),
\end{equation}
and then 
\begin{equation}\label{XHH}
X_H= (\frac{\pa H}{\pa p^i}-b_iR(H))\frac{\pa}{\pa q^i}+
 ( -\frac{\pa H}{\pa q_i}+a_iR(H))\frac{\pa }{\pa p_i}
  +(-a_i\frac{\pa H}{\pa p^i}+b_i\frac{\pa H}{\pa
    q^i}-H)R.
 \end{equation}
The vector field $\grad  H$, i.e.
the solution of the equation
\[
  (\grad  H)^{\flat}=\dd H
  \]
has the expression
\begin{equation}\label{gr}
  \grad H= (\frac{\pa H}{\pa p_i}-b_iR(H))\frac{\pa}{\pa q^i}+
  (-\frac{\pa H}{\pa p_i}+a_iR(H))\frac{\pa}{\pa p_i}
  + (-a_i\frac{\pa H}{\pa p_j}+ b_i\frac{\pa H}{\pa
   q^j}+R(H))R.
\end{equation}
The vector fields $X_H$ \eqref{XHH}  and $\grad H$ \eqref{gr} are related by the relation
\begin{equation}\label{XGR}
  X_H=\grad{H}-(H+R(H))R,\end{equation}
which is similar with \eqref{HAM} but with Reeb vector \eqref{320} instead of
\eqref{reeb3}.

The equations of motion on the  almost cosymplectic manifold
$(M,\theta,\Omega)$
are
\begin{equation}\label{ECABC}
  \frac{\dd q^i}{\dd t}- A_i=0,\quad
\frac{\dd p_i}{\dd t}- B_i=0,\quad
\frac{\dd \kappa}{\dd t}-C  =0,\quad i=1,\dots,n,
\end{equation}
where $(A_i,B_i,C)$ have the expressions given in \eqref{ABC}.

In particular, we find 
\begin{subequations}\label{CENR}
  \begin{align}
    \grad q^i & = -\frac{\pa}{\pa p_i}+b_i R,\\
    \grad p_i & =\frac{\pa}{\pa q^i}-a_iR,\\
    \grad \kappa &= \frac{1}{c}(-b_i\frac{\pa}{\pa
                   q^i}+a_i\frac{\pa}{\pa
                   p_i}+R),\\
    X_{q^i} &= -\frac{\pa}{\pa p_i}+(b_i-q^i)R,\\
    X_{p_i} &=\frac{\pa}{\pa q^i}-(a_i+p_i) R,\\
    X_{\kappa} & =  \frac{1}{c}(-b_i\frac{\pa}{\pa
                 q^i}+a_i\frac{\pa}{\pa p_i}-\kappa \frac{\pa}{\pa \kappa}).
                   \end{align}
  \end{subequations}
\end{Theorem}
\begin{proof} The proof is elementary. We indicate the main steps.

  In
  order to apply \eqref{bemol} we calculate
  \[
    X_H\lrcorner \theta =A_ia_i+B_ib_i+C,\quad
    X_H\lrcorner \Omega=-B_i\dd q^i+A_i\dd p_i,
  \]
  and we find 
 \begin{equation}\label{XFLHH}
   X^{\flat}_H\!=\![-B_i\!+\!a_i(A_ja_j\!+\!B_jb_j\!+\!C)]\!\dd\!
  q^i\!+\![A_i\!+\!b_i(A_ja_j\!+\!B_jb_j\!+\!C)]\!\dd\!p_i\!+\! (A_ja_j\!+\!B_jb_j+\!C)c\dd\!
  \kappa.
  \end{equation}
   We calculate
   \begin{equation}\label{PARDH}
     \dd H-(R\lrcorner H+H)\theta= [\frac{\pa H}{\pa
  q^i}-a_i(R( H)+H)]\dd q^i
+ [\frac{\pa H}{\pa p_i}-b_i(R(H)+H)]
\dd p_i-cH\dd \kappa \end{equation}
and \eqref{ABC} follows from the identification  of \eqref{XFLHH} with
\eqref{PARDH}.

The calculation of $\grad H$ is done
similarly and from \eqref{XHH} and \eqref{gr} it follows \eqref{XGR}.
\end{proof}

Now we compare our results in Theorem \ref{main} with the results of
Albert \cite{Albert}.
\begin{Remark}\label{RR1}
  
The transitive almost contact structure considered by Albert
\cite{Albert} at {\bf{TACS}} in the{\emph{ Appendix}}  corresponds in our case in \eqref{t1}  to
\[
  a_i=\epsilon p_i,\quad b_i=0,\quad  c=1,\quad
  i=1,\dots,n,
\]
and the Reeb vector is 
\[
R(H)=\frac{\pa H}{\pa \kappa}.
\]
\eqref{XHH} becomes
  \[
    X_H
 = \frac{\pa H}{\pa p^i}\frac{\pa}{\pa q^i}+
 ( -\frac{\pa H}{\pa q_i}+\epsilon p^i\frac{\pa H}{\pa \kappa})\frac{\pa }{\pa p_i}-
  (\epsilon p^i\frac{\pa H}{\pa p^i}+ H)\frac{\pa}{\pa \kappa},
  \]
and $ X_H$  verifies the relations 
\begin{equation}\label{HHX}
  X_H\lrcorner \theta =-H, \quad X_H\lrcorner \Omega= \dd
  H-R(H)\theta. \end{equation}

Equation of $\grad f $ becomes
\[
\grad f=\frac{\pa f}{\pa p_i}\frac{\pa}{\pa q^i}+(\epsilon
p_i\frac{\pa f}{\pa \kappa}-\frac{\pa f}{\pa q^i})\frac{\pa }{\pa
  p_i}+(\frac{\pa f}{\pa \kappa}-\epsilon  p_i\frac{\pa f}{\pa
   p_i})\frac{\pa}{\pa \kappa}.
\]

Equations \eqref{CENR} become
\begin{subequations}\label{CENR1}
  \begin{align}
    \grad q^i & = -\frac{\pa}{\pa p_i},\\
    \grad p_i & =\frac{\pa}{\pa q^i}-\epsilon p_iR,\\
    \grad \kappa &= \epsilon p_i\frac{\pa}{\pa
                   p_i}+\frac{\pa}{\pa \kappa},\\
    X_{q^i} &= -\frac{\pa}{\pa p_i}- q^i\frac{\pa}{\pa \kappa}, \label{312d}\\
    X_{p_i} &=\frac{\pa}{\pa q_i}-(\epsilon+1)p_i\frac{\pa}{\pa \kappa},\label{312e}\\
    X_{\kappa} & =  +\epsilon p_i\frac{\pa}{\pa p_i}-\kappa\frac{\pa}{\pa \kappa}.\label{312f}
  \end{align}
  \end{subequations}

  Equations of the vector field $X_f$ at \cite[p 636]{Albert} are in
agreement with our \eqref{ABC},  except
the following  two differences:
\begin{itemize}
\item  the coefficient
  \[\epsilon (f-p_i\frac{\pa}{\pa t})\]
  of $\frac{\pa}{\pa t}$ should be replaced with
\[-\epsilon p^i\frac{\pa f}{\pa p_i}-f.\] This comes from the fact that:
\item  in formulas \cite[(3)]{Albert}  instead of \[X_f\lrcorner \theta
=\epsilon f,\] we should have \[X_f\lrcorner \theta =-f\] as in \eqref{HHX}. With these
correction,  Albert's formulas  for $X_f$ on \emph{  p 636} are particular cases
of our formula \eqref{XHH} and, if $\epsilon =-1$, corrected equation
\emph{(2.12)} in \cite{MLEO} is  find again. Similarly, if  instead of
\eqref{t1}
we consider \eqref{tace}, equations of motion \cite[(2.13), (2.14), (2.15)]{Albert}
in \cite{MLEO} are particular cases of ours equations \eqref{ECABC}.
\item  in formula \eqref{CENR1} Albert gives the different values
for $X_{q^i}$,  $X_{p_i}$ and $X_{\kappa}$ in  \eqref{312d},
\eqref{312e}, respectively \eqref{312f}
\[X_{q^i}=-\frac{\pa}{\pa p_i}+\epsilon q^i\frac{\pa}{\pa t}, \quad
  X_{p_i}=\frac{\pa}{\pa q^i}, \quad
  X_{t}=\epsilon(t\frac{\pa}{\pa t}+p_i\frac{\pa}{\pa p_i}).
\]
\end{itemize}
\end{Remark}

\section{Equations of motion on  $\tilde{\mc{X}}^J_1$}
\subsection{Generalized transitive almost cosymplectic Hamiltonian systems} \label{ACHS}

In the standard approach of CS,
the  $G$-invariant \Ka~ two-form  on a $2n$-dimensional homogenous 
manifold $M=G/H$ is obtained from the \Ka~ potential $f$ via the recipe
\begin{subequations}
  \begin{align*}-\ii\omega_M & =\pa\bar{\pa}f, ~f(z,\bar{z})=\log
  \mc{K}(z,\bar{z}), ~\mc{K}(z,\bar{z}):=(e_{{z}},e_{{z}}),\\
\omega_M(z,\bar{z}) & =\ii \sum_{\alpha,\beta}h_{\alpha\bar{\beta}}\dd
  z_{\alpha}\wedge \dd \bar{z}_{\beta},~
  h_{\alpha\bar{\beta}}=\frac{\pa^2 f}{\pa z_{\alpha}\pa
      \bar{z}_{\beta}},~
                      h_{\alpha\bar{\beta}}=\bar{h}_{\beta\bar{\alpha}},~\alpha,\beta=1,\dots,n,
  \end{align*}
  \end{subequations}
where   $\mc{K}(z,\bar{z})$
is  the scalar product of two  un-normalized Perelomov's  CS-vectors $e_{{z}}$ at
$z\in M$  \cite{sbj,SB15, perG}.

   In accord with  \cite[p 42 ]{ball}, \cite[p 28]{green},
   \cite[Appendix B]{SB19},     the Riemannian metric associated with the Hermitian  metric
    on the manifold $M$ in local coordinates is 
    \begin{equation}\label{asm}\dd
    s^2_{M}(z,\bar{z})=\sum_{\alpha,\beta=1}^nh_{\alpha\bar{\beta}}\dd
    z_{\alpha}\otimes\dd \bar{z}_{\beta}.\end{equation}

From the \Ka~ two-form \eqref{kk1} obtained  via CS in \cite{sbj} we
get with \eqref{asm} the balanced metric  on the Siegel-Jacobi
disk $\mc{D}^J_1$. If in $\dd s^2(w,z)_{\mc{D}^J_1} $ we apply the
second partial Cayley transform \eqref{ULTRAN},  then 
we obtain the invariant metric \eqref{METRS2} on $\mc{X}^J_1$.

We endow   $\tilde{\mc{X}}^J_1$  with a    generalized transitive
 almost  cosymplectic structure, i.e.  an almost cosymplectic
 structure $(M,\theta,\Omega)$ such that
 \[\dd \Omega =0.
 \]
\begin{lemma}\label{L1}
  If we introduce into the \Ka~two-form  $\omega_{\mc{D}^J_1} (w,z)$
  \eqref{kk1} the second partial Cayley transform \eqref{ULTRAN1}
$(w,z)\rightarrow (x,y,q,p), y>0$ we get the symplectic two-form
  \begin{equation}\label{TT123}
    \omega_{\mc{X}^J_1}(x,y,q,p) =\frac{k}{y^2}\dd x\wedge \dd y +2\nu \dd q\wedge \dd
    p,\quad y>0.
    \end{equation}

In the notation of \cite{SB19}, we  introduce on the extended
5-dimensional Siegel-Jacobi half-plane  $\tilde{\got{X}}^J_1$
parametrized in $(x,y,p,q,\kappa)$   the almost
cosymplectic structure  
$(\tilde{\mc{X}}^J_1,\theta,\omega)$, where $\theta = \lambda_6$ and
$\omega$ is \eqref{TT123}, i.e. 
\begin{subequations}\label{TT1}
  \begin{align}\theta & =\sqrt{\delta}(\dd \kappa
  -p \dd q+q\dd p),\quad \delta>0, \label{TT11}\\
    \omega & =\frac{k}{y^2}\dd x\wedge \dd y +2\nu \dd q\wedge \dd
             p,\quad y>0 \label{TT12}.
             \end{align}
           \end{subequations}
We have
           \begin{equation}\label{cond}
             \dd \omega=0,\quad \theta \wedge
             \omega^2=2\frac{k\nu\sqrt{\delta}}{y^2}\dd x\wedge\dd
             y\wedge\dd q\wedge\dd p\wedge \dd \kappa, 
             \end{equation}
             and   $(\tilde{\got{X}}^J_1,\theta,\omega)$ verifies the condition
             \eqref{thtOM} of an almost cosymplectic manifold.

             \eqref{TT11} corresponds in \eqref{t1} to the choice:
\begin{equation}\label{cond1}
  a_1=b_1=0,\quad a_2=-\frac{\sqrt{\delta}}{2\nu} p,\quad
  b_2=\sqrt{\delta}q ,\quad c=\sqrt{\delta}, \quad n=2,
  \end{equation}
  and \eqref{TT12} corresponds in \eqref{tac2} to the choice:
             \begin{equation}\label{cond2}
               q^1=kx,\quad p_1=-\frac{1}{y},\quad  q^2=2\nu
               q,\quad p_2=p,  \quad n=2.
             \end{equation}
             In Darboux coordinates we have a particular almost
             cosymplectic manifold
             $(\tilde{\got{X}}^J_1,\theta,\omega)$ verifying
             \eqref{thtOM} and in  addition the  condition 
             \[
               \dd \omega=0. 
             \]
  $(\tilde{\got{X}}^J_1,\theta,\omega)$         is a   generalized transitive
 almost  cosymplectic  manifold.

\end{lemma}
             
Introducing  \eqref{ABC} in Theorem \ref{main}, 
we get
\begin{Proposition}\label{P16}
The coefficients \eqref{ABC}  of  vector field $X_H$ \eqref{XHH} on
the generalized transitive almost cosymplectic manifold
$(\tilde{\mc{X}}^J_1,\theta,\omega)$ have the expressions
\begin{subequations}\label{ABCEU}
\begin{align}
A_1&=y^2\frac{\pa }{\pa y}, \quad\quad A_2=\frac{\pa H}{\pa p}\color{red}-q\frac{\pa
     H}{\pa \kappa}\color{black},\\
B_1 & =-\frac{1}{k}\frac{\pa H}{\pa x},\quad
      B_2=-\frac{1}{2\nu}(\frac{\pa H}{\pa q}\color{red}+p\frac{\pa H}{\pa \kappa})\color{black},\\
\color{red}C & \color{red}=\frac{1}{2\nu}(p\frac{\pa H}{\pa p}+q\frac{\pa H}{\pa q})-H.\color{black}
\end{align}
\end{subequations}
Equations of motion  \eqref{ECABC} on the 5-dimensional extended
Siegel-Jacobi half-plane  $\tilde{\mc{X}}^J_1$ are 
\begin{subequations}\label{ECXXX}
\begin{align}
\dot{x} &=\frac{y^2}{k}\frac{\pa H}{\pa y}, \quad\qquad\qquad \dot{y} = -\frac{y^2}{k}\frac{\pa H}{\pa x},\\
\dot{q} &= \frac{1}{2\nu}(\frac{\pa H}{\pa p}
          \color{red}-q\frac{\pa H}{\pa \kappa}\color{black}),\quad
\dot{p}  =-\frac{1}{2\nu}(\frac{\pa H}{\pa q}\color{red}+p\frac{\pa
          H}{\pa\kappa}\color{black}),\\
\color{red}\dot{\kappa}& \color{red}=\frac{1}{2\nu}(p\frac{\pa H}{\pa p} +q\frac{\pa
                         H}{\pa q})-H.
 \end{align}
\end{subequations}
\color{black}
\end{Proposition}

\subsection{Contact Hamiltonian systems}\label{CHS}

We recall that in \cite{SB19} we noticed  that for $\eta =\lambda_6$
we have $\dd \eta^2 =0$ and if we want to introduce a contact
structure on $\tilde{\mc{X}}^J_1$, we have to choose 
$\eta\not= \lambda_6$.

Now we organize $\tilde{\mc{X}}^J_1$ as a contact Hamiltonian system in the
sense of \cite[\S~2.3]{MLEO}. Then we apply  Proposition \ref{PP19} to
determine the associated almost contact metric structure:
\begin{Proposition}\label{PPP}
  Let $\eta_0$ be such that $\dd \eta_0=\omega_{\mc{X}^J_1}$,  where $\omega_{\mc{X}^J_1}$ is
  given by 
  \eqref{TT123}. We chose
\begin{equation}\label{ETA00}
  \eta_o=\sqrt{\delta}\dd \kappa +\frac{k}{y}\dd x+\nu(-p\dd q+q\dd p).
\end{equation}
 $\eta_0\wedge \dd \eta_0^2$ has the same nonzero value as  the volume
 form given by the second equation \eqref{cond}, \eqref{thtOM1} is
 verified, and   $(\tilde{\mc{X}}^J_1, \eta_0)$ is a contact manifold.

 a) For  \eqref{ETA00} chosen  as $\eta_0$  and $ g_{\tilde{\mc{X}}^J_1}$
 given by \eqref{begGG}, there is no $\Phi$ 
 so  that $ (\tilde{\mc{X}}^J_1,
  \Phi,\xi,\eta_0,$        $g_{\tilde{\mc{X}}^J_1})$
  is an almost contact metric structure, where $\xi$ is given by \eqref{csi}, \eqref{IARXI}.

  b) The contact manifold $(\tilde{\mc{X}}^J_1,\eta_0)$ can be endowed with an  almost contact
  metric structure $ (\tilde{\mc{X}}^J_1,
  \Phi,\xi,\eta_0,g'_{\tilde{\mc{X}}^J_1})$. 
  The six components  $\Phi_{xx},
 \Phi_{xy} ,\Phi_{xq}, \Phi_{xp}, \Phi_{yx}, \Phi_{qq}$  of $\Phi$  in \eqref{phiind}
 can be expressed as function of the four remaining independent variables
 $\Phi_{yq},\Phi_{yp},\Phi_{qp},\Phi_{pq}$, while the rest of the
 components of $\Phi$ are obtained with \eqref{4colphi}, \eqref{REF}, \eqref{4reln}.
The fundamental quadratic form associated to the metric tensor  $g'$  \eqref{GVAL1} 
  must  be
 positive definite.
\end{Proposition}

\begin{proof}
  a) We use \eqref{gphi} for $\hat{\Phi}$ corresponding to $\eta_0$
  and \eqref{begGG}  written as 
    \begin{equation}\label{SS1}
g_{\tilde{\mc{X}}^J_1}=\left(\begin{array}{cc} A &0 \\0
                                                 &B\end{array}\right),~~
                                             A=
                                             \left(\begin{array}{cc}
                                                     g_{xx} & 0\\0 &
                                                                     g_{yy}\end{array}
                                                                 \right),
                                                                 ~B
                                                                 =\left(\begin{array}{ccc}
                                                                          g_{qq}
                                                                          &g_{qp}&
                                                                                   g_{q\kappa}\\
                                                                        g_{pq}&
                                                                                g_{pp}&
                                                                                        g_{p\kappa}\\
                                                                        g_{\kappa
                                                                          q}&g_{\kappa
                                                                              p}&
                                                                                  g_{\kappa \kappa}\end{array}\right).
    \end{equation}
We write down $\Phi$ as
\begin{equation}\label{SS2}
  \Phi\!=\! \left(\begin{array}{cc}\Phi_1  &\Phi_2\\ \Phi_3&
                                                         \Phi_4\end{array}\right),
                                                     \text{~~~}\Phi_1\in
                                                     M(2,2,\R),~\Phi_4\in
                                                     M(3,3,\R),~\Phi_2,~\Phi_3^t\in M(2,3,\R).
      \end{equation}
 With \eqref{SS1} and \eqref{SS2} we get
  \[
  g_{\tilde{\mc{X}}^J_1}\Phi=\left(\begin{array}{cc}
  A\Phi_1& A\Phi_2\\B\Phi_3&B\Phi_4\end{array}\right).
 \]
We write    down $\hat{\Phi}$  as
      \begin{equation}\label{SS3}
        \hat{\Phi}= \left(\begin{array}{cc}\hat{\Phi}_1  & 0 \\ 0 &
   \hat{\Phi}_4\end{array}\right),
 ~\hat{\Phi}_1 =\left(\begin{array}{cc} 0& \tu\\-\tu
 &0\end{array}\right), ~\hat{\Phi}_4=\left(\begin{array}{ccc} 0&\sg&
                                                                      0\\
                                             -\sg & 0 &0\\
                                           0& 0& 0\end{array}\right),
                                       ~\tu=\frac{k}{y^2},~\sg=2 \nu .
                                     \end{equation}
  With \eqref{gphi} we find the relations 
  \begin{equation}\label{DR1}
    \Phi_1 =\left(\begin{array}{cc} 0 &\frac{\tu}{g_{xx}}\\
 -\frac{\tu}{g_{yy}} &0\end{array}\right), ~\Phi_2=0,~\Phi_3=0,
~\Phi_4=\left(\begin{array}{ccc} g_{qq}& g_{qp}& 0\\
                g_{pq}& g_{pp}& 0\\
                g_{\kappa q}& g_{\kappa p} & 0\end{array}\right).\\
                                           \end{equation}
 With \eqref{DR1}, the first
 condition \eqref{679}  implies
 \[
   \xi^t=(0,0,0,0,\xi_{\kappa}),
 \]
while the  second  condition \eqref{679} gives
  \begin{equation}\label{csi}
  \xi=\frac{1}{\sqrt{\delta}}\frac{\pa}{\pa \kappa}.
\end{equation}
The second condition \eqref{679} implies
\begin{subequations}\label{NOT}
\begin{align}
0 = &
     \frac{k}{y}\Phi^x_x-\nu\Phi^q_x+\nu\Phi^p_x+\sqrt{\delta}\Phi^{\kappa}_x,\\
  0 = & \frac{k\tu}{yg_{xx}},\label{NOTb}\\
  0 = & \dots   .
\end{align}
\end{subequations}
But \eqref{NOTb} can't  be satisfied 
and a) is proved.

b)  With \eqref{ACM}, \eqref{679}, \eqref{680}, for  that the
contact manifold $(M_{2n+1},\eta)$ to have an almost contact
metric structure $(M,\Phi,\xi,\eta,g)$, the following equations must
be satisfied 
\begin{subequations}\label{EQQ} 
  \begin{align}
     & \text{rank~} \Phi= 2n, \label{EC8}\\
    &\eta\lrcorner \xi=1, \quad \xi^t,\eta\in M(1,n,\R)\label{EC1},\\
     &  \Phi \xi=0,\label{EC3}\\
    & \eta \Phi =0,\label{EC4} \\
    & \Phi^2  =-\mb{1}_n +\xi\otimes\eta ,\label{EC2}\\
    &  \eta=g \xi,\label{EC6}\\
    & g= \eta\otimes \eta -\hat{\Phi}\Phi,\label{EC7}\\
    & \hat{\Phi}\xi=0\label{EC5}.
   \end{align}
 \end{subequations}
 We recall that $\hat{\Phi}$ is defined in \eqref{gphi} and is connected with  $\dd
 \eta$ as in \eqref{ECSA2}.
 Equations \eqref{EC8}-\eqref{EC7} appear in \cite{SAS60} (and also
 in
 \cite{sh}),  where it is underlined that the first 5 relations  are not independent. 
   Equation \eqref{EC5} is a consequence of \eqref{EC4},
   \eqref{gphi}. See also Theorem 3.1 in \cite{sas}.

   Below we denote the components  $\Phi^i_j$ of the (1,1)-tensor
   field $\Phi$ by $\Phi_{ij}$, i.e.
   \begin{equation}\label{below}
     \Phi=(\Phi)_{ij},\quad i,j=x,y,q,p,\kappa.
     \end{equation}

   1) With $\eta_0$ given by \eqref{ETA00} and \eqref{EC1}, we get for
   $\xi=(\xi_x,\xi_y,\xi_q,\xi_p,\xi_{\kappa})$ 
\[\frac{k}{y}\xi_x-\nu p\xi_q+\nu q\xi_p+\sqrt{\delta}\xi_k=1.\]
   With \eqref{EC5}  for $\hat{\Phi}$ given by \eqref{SS3}  we get
   again for $\xi$ the form \eqref{csi}
\begin{equation}\label{IARXI}
  \xi^t=(0,0,0,0,\frac{1}{\sqrt{\delta}}).
\end{equation}

2) With \eqref{EC3} and \eqref{IARXI} we get the relations 
\begin{equation}\label{4colphi}
  \Phi_{z\kappa}=0, \quad  z=x,y,q,p,\kappa.
\end{equation}

3) With \eqref{EC4}, we get
\begin{equation}\label{4rel}
    \Phi_{\kappa z} =-\frac{1}{\sqrt{\delta}}(\frac{k}{y}\Phi_{xz}-\nu p\Phi_{q
      z}+\nu q \Phi_{pz}),\quad z=x,y,q,p.
        \end{equation}

        4) With the values \eqref{ETA00} for $\eta$ and \eqref{IARXI} for $\xi$
        introduced in \eqref{EC6}, we get
        \begin{equation}\label{422}
          (g_{ x\kappa}, g_{y \kappa},      g_{q \kappa},  g_{p
            \kappa},  g_{\kappa \kappa})=(\frac{k\sqrt{\delta}}{y},0,-\nu\sqrt{\delta}p,\nu\sqrt{\delta}q,\delta).       
          \end{equation}
   
  5) Introducing \eqref{SS3} in  \eqref{EC7}, with \eqref{4colphi}, \eqref{422} we
  get
  \begin{equation}\label{GVAL}
    \begin{aligned}
    g & =\left(\begin{array}{ccccc}g_{xx}&g_{xy}&g_{xq}&g_{xp}&g_{x\kappa}\\
                                      &g_{yy}&g_{yq}&g_{yp}&g_{y\kappa}\\
                                      & &g_{qq}&g_{qp}&g_{q\kappa}\\
                                      & &&g_{pp}&g_{p\kappa}\\
               & &&&g_{\kappa\kappa}
               \end{array}\right) \\ ~~&= \left(\begin{array}{ccccc}
  \frac{k^2}{y^2}-\tu\Phi_{xx}&\tu\Phi_{xx}&
 -\frac{\nu kp}{y}-\tu\Phi_{yq}& \frac{\nu kq}{y}-\tu\Phi_{yp}&\frac{k\sqrt{\delta}}{y}\\
                                      &\tu\Phi_{xy}&\tu\Phi_{xq}&\tu\Phi_{xp}&0\\
                                      & &\nu^2p^2-\sg\Phi_{pq}&-\nu^2pq+\sg\Phi_{pp}&-\nu\sqrt{\delta}p\\
                                      & &&\nu^2q^2+\sg\Phi_{qp}&\nu\sqrt{\delta}q\\
               & &&&\delta
                   \end{array}\right),
  \end{aligned}               
\end{equation}
where we have not written down the matrix elements under the
diagonal of the symmetric matrix $g$.

6) But the symmetry of the matrix  $g$
\eqref{GVAL} imposes  the following restrictions on  components of the 
$(1,1)$-tensor $\Phi$:
\begin{subequations}\label{REF}
  \begin{align}
    \Phi_{yy} = &-\Phi_{xx},\\
    \Phi_{qx} = & -\zeta\Phi_{yp},\quad \zeta:=\frac{\tu}{\sg},\\
   \Phi_{qy} = & \zeta\Phi_{xp}, \\
   \Phi_{px} = &  \zeta \Phi_{yq}, \\
    \Phi_{pp} = &- \Phi_{qq}, \\
    \Phi_{py} = &-  \zeta \Phi_{xq}.
  \end{align}
\end{subequations}
We  choose  as  independent components  of tensor field $\Phi$  the 
components $\Phi$ minus the l.h.s. components of \eqref{REF}, i.e. the
submatrix of $\Phi$ with the 
following 10 components 
\begin{equation}\label{phiind}
\left(\begin{array}{ccccc}\Phi_{xx}&\Phi_{xy}&\Phi_{xq}&\Phi_{xp}&\\
                                   \Phi_{yx}   & &\Phi_{yq}&\Phi_{yp}&\\
                                      & &\Phi_{qq}&\Phi_{qp}& \\
                              & &  \Phi_{pq}& &\\
                              & &&& \end{array}\right) .
\end{equation}
With \eqref{REF}, the relations \eqref{4rel} became 
\begin{subequations}\label{4reln}
  \begin{align}
    \Phi_{\kappa x} & =-\frac{1}{\sqrt{\delta}}(\frac{k}{y}\Phi_{xx}+\nu\zeta p\Phi_{yp}
      +\nu \zeta q \Phi_{yq}),\\
     \Phi_{\kappa y} & =-\frac{1}{\sqrt{\delta}}(\frac{k}{y}\Phi_{x
                       y}-\nu \zeta p\Phi_{xp}
                    - \nu\zeta q \Phi_{xq}),\\
   \Phi_{\kappa q } & =-\frac{1}{\sqrt{\delta}}(\frac{k}{y}\Phi_{x
                       q}-\nu p\Phi_{q q}+\nu q \Phi_{pq }),\\
    \Phi_{\kappa p} & =-\frac{1}{\sqrt{\delta}}(\frac{k}{y}\Phi_{xp}-\nu p\Phi_{q
                   p}-\nu q \Phi_{qq }). 
   \end{align}
   \end{subequations}\newcommand{\te}{\ensuremath{\tau}}

   With \eqref{REF}, equation \eqref{GVAL} becomes
\begin{equation}\label{GVAL1}
    g'  = \left(\begin{array}{ccccc}
  \frac{k^2}{y^2}-\tu\Phi_{xx}&-\tu\Phi_{yy}&
 -\frac{\nu kp}{y}-\tu\Phi_{yq}& \frac{\nu kq}{y}-\tu\Phi_{yp}&\frac{k\sqrt{\delta}}{y}\\
                                      &\tu\Phi_{xy}&\tu\Phi_{xq}&\tu\Phi_{xp}&0\\
                                      & &\nu^2p^2-\sg\Phi_{pq}&-\nu^2pq-\sg\Phi_{qq}&-\nu\sqrt{\delta}p\\
                                      & &&\nu^2q^2+\sg\Phi_{qp}&\nu\sqrt{\delta}q\\
               & &&&\delta
\end{array}\right) .
 \end{equation}
Taking into account \eqref{4colphi},   \eqref{4rel}, \eqref{REF}, \eqref{4reln}, the 25 equations \eqref{EC2} for the tensor
 $\Phi$ in the convention \eqref{below} expressed as function of   the 10 independent components
 \eqref{phiind}  are reduced only  to the following 6 independent equations
\begin{subequations}\label{pHIp}
\begin{align}
  & \Phi^2_{xx}+\Phi_{xy}\Phi_{yx}+\zeta(\Phi_{xp}\Phi_{yq}-\Phi_{xq}\Phi_{yp})=-1,\label{612a}\\
  & \Phi_{xq}(\Phi_{xx}+\Phi_{qq})+\Phi_{xy}\Phi_{yq}+\Phi_{xp}\Phi_{pq}=0,\label{612b}\\
  & \Phi_{xp}(\Phi_{xx}-\Phi_{qq})+\Phi_{xy}\Phi_{yp}+\Phi_{xq}\Phi_{qp}=0,\label{612c}\\
  & \Phi_{yq}(\Phi_{qq}-\Phi_{xx})+\Phi_{yx}\Phi_{xq}+\Phi_{yp}\Phi_{pq}=0,\label{612d}\\
  & -\Phi_{yp}(\Phi_{xx}+\Phi_{qq})+\Phi_{yx}\Phi_{xp}+\Phi_{yq}\Phi_{qp}=0,\label{612e}\\
  & \zeta(\Phi_{xp}\Phi_{yq}-\Phi_{yp}\Phi_{xq})+\Phi^2_{qq}+\Phi_{qp}\Phi_{pq}=-1.\label{612f}
\end{align}
\end{subequations}
  In \eqref{pHIp}  the  first     3  (respectively, next 2, last,
  no, no) equations are the only independent equations
of the product of    the first (respectively, second, third, fourth, fifth) line of
$\Phi$ with $\Phi$  in \eqref{EC2}.

 7) The ten independent variables $\Phi$ \eqref{phiind} verify the 6
 independent equations \eqref{pHIp}. We shall choose 6 independent
 variables in the set \eqref{phiind} and express the four remaining
 variables using  six  independent equations \eqref{pHIp}.

 From equations \eqref{612b} and \eqref{612c}, we get  if
 $\Phi_{xq}\Phi_{xp}\not= 0$
 \begin{subequations}\label{QQXX}
   \begin{align}
     \Phi_{xx}+\Phi_{qq} &
                           =-\frac{1}{\Phi_{xq}}(\Phi_{xy}\Phi_{yq}+\Phi_{xp}\Phi_{pq}), \label{QQXX1}\\
     \Phi_{xx}-\Phi_{qq} &
                           =-\frac{1}{\Phi_{xp}}(\Phi_{xy}\Phi_{yp}+\Phi_{xq}\Phi_{qp}). \label{QQXX2} 
   \end{align}
 \end{subequations}
 Introducing \eqref{QQXX2} into  \eqref{612d}, we get if
 $\Phi_{xq}\Phi_{xp}\not= 0$
 \begin{equation}\label{PYX}
   \Phi_{yx}=-\frac{1}{\Phi_{xq}\Phi_{xp}}[\Phi_{yq}(\Phi_{xy}\Phi_{yp}+\Phi_{xq}\Phi_{qp})+\Phi_{yp}\Phi_{pq}\Phi_{xq}].
   \end{equation}
 From \eqref{QQXX}, we get  if
 $\Phi_{xq}\Phi_{xp}\not= 0$
 \begin{subequations}\label{QQXX11}
   \begin{align}
     2\Phi_{xx} & =
                   -\frac{1}{\Phi_{xq}}(\Phi_{xy}\Phi_{yq}+\Phi_{pq}\Phi_{xp})
                   -\frac{1}{\Phi_{xp}}(\Phi_{xy}\Phi_{yp}+\Phi_{xq}\Phi_{qp}),\\
      2\Phi_{qq} & = -\frac{1}{\Phi_{xq}}(\Phi_{xy}\Phi_{yq}+\Phi_{pq}\Phi_{xp}) +\frac{1}{\Phi_{xp}}(\Phi_{xy}\Phi_{yp}+\Phi_{xq}\Phi_{qp}).
     \end{align}
   \end{subequations}
   But from  \eqref{612a}, \eqref{612f} we get
   \begin{equation}\label{615}
     \zeta(\Phi_{xp}\Phi_{yq}-\Phi_{xq}\Phi_{yp})+1=-\Phi^2_{xx}-\Phi_{xy}\Phi_{yx}=-\Phi^2_{qq}-\Phi_{qp}\Phi_{pq}.
   \end{equation}
   Introducing \eqref{PYX} into \eqref{612e}, we get
   \[
     \Phi_{xy}[\Phi_{xq}(\Phi_{yx}\Phi_{xp}+\Phi_{qp}\Phi_{yq})+\Phi_{yp}(\Phi_{xp}\Phi_{pq}+\Phi_{yq}\Phi_{xy})]=0.
     \]
   If $\Phi_{xy}\not=0$, we get
    \begin{equation}\label{PYX1}
   \Phi_{yx}=-\frac{1}{\Phi_{xq}\Phi_{xp}}[\Phi_{yq}(\Phi_{xy}\Phi_{yp}+\Phi_{xq}\Phi_{qp})+\Phi_{yp}\Phi_{pq}\Phi_{xp}].
 \end{equation}
 But comparing \eqref{PYX} with \eqref{PYX1}, we get
 \begin{equation}\label{QU}
   \Phi_{xp}=\Phi_{xq}.
 \end{equation}
 
 With \eqref{QQXX11},  \eqref{QU}, equations \eqref{612a}, 
 \eqref{612f} became
 \begin{subequations}\label{QQXX12}
   \begin{align}
     2\Phi_{xx} & =
                   -\frac{1}{\Phi_{xq}}[\Phi_{xy}(\Phi_{yq}+\Phi_{yp})+\Phi_{xq}(\Phi_{pq}+\Phi_{qp})],\label{1Q2}\\
     2\Phi_{qq} & = \frac{1}{\Phi_{xq}}[\Phi_{xy}(-\Phi_{yq}+\Phi_{yp})+\Phi_{xq}(-\Phi_{pq}+\Phi_{qp})].\label{2Q2}
     \end{align}
   \end{subequations}
 With \eqref{QU}, equations \eqref{615} became
 \begin{subequations}\label{MRR}
   \begin{align}
     \Phi_{xx}^2+\zeta\Phi_{xq}(\Phi_{yq}-\Phi_{yp})+1& =
              -\Phi_{xy}\Phi_{yx},\label{MRR1}\\
      \Phi_{qq}^2 +\zeta\Phi_{xq}(\Phi_{yq}-\Phi_{yp})+1& = -\Phi_{qp}\Phi_{pq}. \label{MRR2}
   \end{align}
 \end{subequations}
 From       \eqref{MRR1} (\eqref{MRR2}) we get with \eqref{1Q2},
 \eqref{PYX1} (\eqref{2Q2})   the value $\Phi_{xy}$  (respectively
 $\Phi_{xq}$).

 We have shown that the six components  $\Phi_{xx},
 \Phi_{xy} ,\Phi_{xq}, \Phi_{xp}, \Phi_{yx}, \Phi_{qq}$  of $\Phi$  in \eqref{phiind}
 can be expressed as function of the remaining independent variables
 $\Phi_{yq},\Phi_{yp},\Phi_{qp},\Phi_{pq}$, while the rest of the
 components of $\Phi$ are obtained with \eqref{REF}.

 Once the (1,1)-tensor $\Phi$ is known, it should be verified that the
 condition \eqref{EC8} is fulfilled. Then    the metric
 tensor $g'$ is   determined with \eqref{GVAL1} and we have to impose the condition that
 the fundamental quadratic form associated to the tensor  $g'$ must  be
 positive definite.

 \end{proof}

\begin{Remark}\label{RR223}
In order to compare the metric matrices \eqref{begGG}  and
\eqref{GVAL1} on ${\tilde{\mc{X}}^J_1}$ we have to make the replacement
$k,\nu\rightarrow \sqrt{k},\sqrt{\nu} $ in \eqref{begGG} due to  the
parametrization \eqref{PARO} of the one-forms in \cite{SB19} and we
get instead of \eqref{begGG} 
\begin{equation}\label{linvG1}
    \begin{split}
 {\rm d} s^2_{\tilde{\mc{X}}^J_1}(x,y,p,q,\kappa)  & =\nu(p\dd
 q^2+q\dd p^2-2pq\dd q\dd p)+\delta \dd
 \kappa^2 + 2\frac{\sqrt{k\delta}}{y}\dd x\dd \kappa \\ & +2\frac{\sqrt{k\nu}}{y}(-p\dd p+q \dd q)\dd x
-2\sqrt{\nu\delta}(p\dd q+
   q\dd p)\dd \kappa
 \end{split}
  \end{equation}
  in the convention \eqref{AK}.
  
  Also it would be interesting to  check that the metric  
  $g_{\tilde{\mc{X}}^J_1}(x,y,p,q,\kappa)$
  \eqref{GVAL1}  can not be obtained from a potential  type
  \eqref{ggold}. 
\end{Remark}

\begin{Proposition}\label{PRR4}
  $(\tilde{\mc{X}}^J_1,\eta_0)$ is a contact Hamiltonian system as
in {\emph{Appendix }}, {\bf C}.

Using the same parametrization  \eqref{cond1} and \eqref{cond2}  in
{\emph {Section \ref{ACHS}}}, we get for $X_H$  \eqref{517} the expression
\begin{equation}\label{XZH}
  \begin{split}
  X_H & =\frac{y^2}{k}\frac{\pa H}{\pa y}\frac{\pa }{\pa x}+
  \frac{1}{2\nu} \frac{\pa H}{\pa p}\frac{\pa}{\pa q}+
  (-\frac{y^2}{k}\frac{\pa H}{\pa x}+\color{green}y\frac{\pa
    H}{\pa \kappa}\color{black})\frac{\pa}{\pa y}\\
 &  -(\frac{1}{2\nu}\frac{\pa H}{\pa q}+\color{green}p\frac{\pa
    H}{\pa \kappa}\color{black})\frac{\pa }{\pa p}
      +\color{green}(-y\frac{\pa H}{\pa y}+p\frac{\pa H}{\pa
    p}-H)\frac{\pa}{\pa \kappa}\color{black}.
    \end{split}
                    \end{equation}
The equations of motion associated to the vector field $X_H$
\eqref{XZH} on the extended Sigel--Jacobi upper half-plane organized
as the contact manifold $(\tilde{\mc{X}}^J_1,\eta_0)$ are 
\begin{subequations}\label{eq2}
  \begin{align}
                      \dot{x} &=\frac{y^2}{k}\frac{\pa H}{\pa y},\quad \dot{y} =  -\frac{y^2}{k}\frac{\pa H}{\pa x}+\color{green}y\frac{\pa
    H}{\pa \kappa}\color{black},\\
                         \dot{q} &=\frac{1}{2\nu} \frac{\pa H}{\pa p}, \quad 
\dot{p} =   -(\frac{1}{2\nu}\frac{\pa H}{\pa q}+\color{green}p\frac{\pa
    H}{\pa \kappa}\color{black}), \\                 
\color{green}\dot{\kappa}   & = \color{green}(-y\frac{\pa H}{\pa y}+p\frac{\pa H}{\pa
    p}-H)\frac{\pa H}{\pa \kappa} \color{black}  . 
  \end{align}
  \end{subequations}
 
The Jacobi bracket \eqref{PAR} on   $\tilde{\mc{X}}^J_1$ is 
\[
  \{f,g\} =\{f,g\}_P+f_e\frac{\pa g}{\pa \kappa}-g_e\frac{\pa f}{\pa \kappa},
\]
where the Poisson bracket \eqref{peste} with the convention
\eqref{cond2} reads  
\[\{f,g\}_P=\frac{1}{k}\frac{y^2+1}{y^2}(\frac{\pa f}{\pa x}\frac{\pa
    g}{\pa y}-\frac{\pa g}{\pa x}\frac{\pa f}{\pa y})+
  \frac{1}{2\nu}[\frac{\pa f}{\pa q}\frac{\pa g}{\pa p} -\frac{\pa
  g}{\pa q}\frac{\pa f}{\pa p} +\frac{1}{y^2}(\frac{\pa f}{\pa
  q}\frac{\pa g}{\pa y}-\frac{\pa g}{\pa q}\frac{\pa f}{\pa y})].
\]
while the Euler operator \eqref{FE} is 
\[f_e=f+\frac{1}{y^3}\frac{\pa f}{\pa y}-p\frac{\pa f}{\pa p}.
\]
\end{Proposition}
\begin{Remark}\label{R333}
  If in the  equations of motion expressed in the S-variables $(x,y,p,q,\kappa)$ 
in \eqref{ECXXX}        (respectively \eqref{eq2}) on the generalized transitive almost
cosymplectic manifold  $(\tilde{\mc{X}}^J_1,\theta,\omega)$
(respectively, on the contact manifold $(\tilde{\mc{X}}^J_1,\eta_0)$)
we ignore the ``red'' (respectively  ``green'') parts we get the equations
of motion on $\mc{X}^J_1$ in $(x,y,p,q)$. 
\end{Remark} 
\subsection{Linear Hamiltonian in the generators of the Jacobi  group  $G^J_1(\R)$}\label{43}

In \cite[(4.7)]{FC} we have considered a linear Hermitian Hamiltonian $\mb{H}$
in the generators  of the Jacobi group $G^J_1$
\begin{equation}\label{LH}
  \mb{H}=\epsilon_a\mb{a}+\bar{\epsilon}_a\mb{a}^{\dagger}+\epsilon_0\mb{K}_0+\epsilon_+\mb{K}_+
  +\epsilon_-\mb{K}_-,
  \quad \bar{\epsilon}_+=\epsilon_-,\quad \bar{\epsilon}_0=\epsilon_0.
\end{equation}
We use  the notation introduced in \cite[\S~4.1 3.]{FC}
\[\epsilon_a:=a+\ii b,~\epsilon_+:=m-\ii n, ~\epsilon_0:=2c,\quad
  a,b,c,m,n \in\R.
\]
 We have proved in \cite[(4.29) in \S~4.3]{FC} that the energy function $\mc{H}$ associated to the
 linear Hamiltonian \eqref{LH} expressed in the variables $(\eta, v)$
 splits into  the sum of two independent functions
\begin{equation}\label{hsum}
  \mc{H}(\eta,v)=\mc{H}(\eta)+\mc{H}(v),\quad v=x+\ii y, ~y>0,~\eta=q+\ii p, 
\end{equation}
where
\begin{subequations}\label{sup}
\begin{align}
\mc{H}(q,p) &=\nu[(m+c)q^2+(c-m)p^2+2nqp+2(aq+bp)],\\
\mc{H}(x,y) & =k\{ \frac{1}{y}[(m+c)(x^2+y^2)-2(nx+cy)]+3c-m\}. 
\end{align}
\end{subequations}
We have underlined in  \cite[Remark 1]{SB20}, \cite[Proposition
3]{SB21} the connection of parameter $\eta=q+\ii p$ which appear in the
FC-transform with the S-variables $(p,q)$ which parametrize the Jacobi
group $G^J_1(\R)$.

Now we particularize equations \eqref{ECXXX} to the  linear
Hamiltonian \eqref{hsum}  to which we add a term  ${h}(\kappa)$
\begin{equation}\label{HPARL}
  \mc{H}=\mc{H}(p
 ,q)+\mc{H}(x,y)+ {h}(\kappa),
\end{equation}
and we get
\begin{Proposition}\label{PR6}
  The equations of motion \eqref{ECABC} on the extended
   Siegel-Jacobi upper half-plane  organized as generalized transitive
   almost  cosymplectic
   manifold $(\tilde{\mc{X}}^J_1,\theta,\omega)$ corresponding to the energy
   function \eqref{HPARL} are 
   \begin{subequations}\label{corH}
  \begin{align}
    \dot{x} & = (c+m)(-x^2+y^2)+mx-c+m,\quad
    \dot{y} =   -2(c+m)y^2+2ny,\label{corHxy}\\
    \dot{q} & =\!-\!(m+c)q-\!np-\!a\color{red}
   \!-\!\frac{q}{2\nu}\frac{\pa{h}}{\pa \kappa}\color{black}, \quad
    \dot{p}   = qn +(c-m)p+b\color{red}-\frac{p}{2\nu}\frac{\pa
              h}{\pa \kappa}\color{black},\label{corHpq}\\
   \color{red}\dot{\kappa} & \color{red}=
  (c+m)qa^2+(-c+m)p^2+(m-n)pq+nq+bp\!-\!\frac{1}{\sqrt{\delta}}
                      h.\color{black}
    \end{align}
  \end{subequations}
\end{Proposition}
\begin{Remark}\label{RR2}

  If in the equations of motion \eqref{corH} on $\tilde{\mc{X}}^J_1$ 
generated by the linear Hamiltonian  \eqref{HPARL}  we ignore
the ``red parts'', the well-known matrix   Riccati equation \eqref{corHxy}  in $(x,y)$ 
and the linear system of differential equations in $(p,q)$ generated by
  the linear Hamiltonian  \eqref{hsum}, \eqref{sup} on $\mc{X}^J_1$
  found
  in \cite{FC} are obtained.
\end{Remark}

\section{Appendix -- a breviar of terminology}\label{APP}

In this section are  collected    definitions of the main
geometric structures  used  in
our  paper. We use the following abbreviations:\\
SH: symplectic Hamiltonian;  ~~~ COS:~ cosymplectic;~~ ACOS:~ almost
cosymplectic,\\~~ GTACOS: ~generalized transitive almost
cosymplectic,~~~TACS:~ transitive almost contact structure,~~~CH:
contact Hamiltonian,~~~C: contact, ~~~ ACM: ~almost contact
metric,~~~SAC: (strict)
almost contact,~~~ N:~ normal,~~~SAS:~ Sasakian,~ACOK: almost coK\"ahler.

\begin{center}
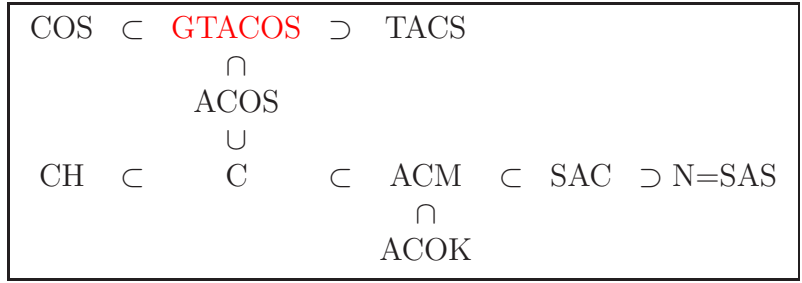
 Table:  Geometric structures on odd dimensional manifolds
\end{center}
\[ 
\boxed{
  \begin{array}{cc cc cc cc cc}
     \text{COS}&\subset&\color{red}\text{GTACOS}\color{black}&\supset&\text{TACS}&&&\\
   &&\cap&&&&&\\
 & &\text{ACOS}& &&&&\\
     &&\cup&&&&&\\
  \text{CH}&\subset&\text{C}&\subset&\text{ACM}&\subset&\text{SAC}&\supset \text{N=SAS}\\
     &&&&\cap&&&\\
   &&&&\text{ACOK}&&&
 \end{array}
 }
\]

{\bf SH} 

In  \cite[Chapter III, Symplectic manifolds and Poisson
manifolds]{paulM}: {\it a symplectic  Hamiltonian system} --  is the pair
 $(M,\Omega)$, where  $\dim M=
2n$, $n\in\N$,  $\Omega$ is a closed non-degenerate two-form, and 
\[\Omega^n\neq 0.\]

If
$H:M\rightarrow \R$ is a Hamiltonian function, then the {\it Hamiltonian
vector field} $X_H$,  $\grad  H$, is the solution of the equation
\begin{equation}\label{GEN}
\flat(X_H)=\dd H, \quad {\text{where~~~}} \flat: TM\rightarrow T^*M,
\quad \flat(X)=X\lrcorner
\Omega.
\end{equation}

In canonical  {\it Darboux coordinates} $(q^i,p_i),~ i=1,\dots,n$ , we have 
\begin{equation}\label{tac2}
  \Omega  = \dd q^i\wedge \dd p_i,
  \end{equation}
  \begin{equation}\label{HXHH}
    X_H=\frac{\pa H}{\pa p_i}\frac{\pa}{\pa q^i}-\frac{\pa
        H}{\pa q^i}
\frac{\pa }{\pa p_i},    \end{equation}
and Hamilton equations of motion are
$$\dot{q}^i=\frac{\pa H}{\pa p_i}=\{q^i,H\}_P,\quad
\dot{p}_i=-\frac{\pa H}{\pa q^i}=\{p_i,H\}_P,$$
where the Poisson bracket $ \{f,g\}_P$  in Darboux  coordinates is 
\begin{equation}\label{peste}
  \{f,g\}_P=\Omega(X_f,X_g)=\frac{\pa f}{\pa q^i}\frac{\pa g}{\pa p_i}-\frac{\pa
    g}{\pa q^i}\frac{\pa f}{\pa p_i},\quad   f,g\in
  C^{\infty} (M).
\end{equation}

Note that if in \eqref{peste} we make the change of coordinates
$(q,p)\rightarrow (p,q)$, then $ \{f,g\}_P\rightarrow - \{f,g\}_P$.

{\bf ACOS} 

In \cite{paul}:      {\it  almost cosymplectic manifold} -- is the
triplet $(M,\theta,\Omega) $,
where $M$ is a $(2n+1)$-dimensional manifold, $\theta\in \got{D}^1$,
$\Omega$ is a   2-form with $\text{rank}(\Omega)=2n$,   and
\begin{equation}\label{thtOM}
  \theta\wedge\Omega^n\neq 0.
\end{equation}

The {\it Reeb vector} $R\in \got{D}^1$ is defined by the equations
\begin{equation}\label{reeb}
R\lrcorner\Omega=0, \quad  
  R\lrcorner\theta=1. 
\end{equation}

 In \cite{Albert}: the almost cosymplectic manifold
$(M,\theta,\Omega) $ of  \cite{paul}  is called {\it almost contact manifold}. \\
It is proved in \cite[Proposition  1]{Albert} that {\it  the application} $\flat: TM\rightarrow  TM^*$
\begin{equation}\label{bemol}
X\rightarrow X^{\flat}=X\lrcorner \Omega + (X\lrcorner\theta)\theta
\end{equation}
{\it is a vector bundle isomorphism}.

{\bf COS}     

In \cite{paul}:   {\it  cosymplectic manifold} -- is the triplet  $(M,\theta,\Omega) $,
defined in  {\bf ACOS},  where

\[
  \dd \theta =0, \quad \dd \Omega=0.
\]

In  \cite{MLEO}: the isomorphism $\flat$ for the  cosymplectic
manifold     $(M,\eta,\Omega)$ 
is defined as in \eqref{bemol}.

The  Darboux  coordinates are $(z,q^i,p_i,)$, $ i=1,\dots,n$, $\Omega$
is defined in \eqref{tac2}, and 
\[
  \eta=\dd z.
  \]
The {\it gradient vector field} $\grad H$, the {\it Hamiltonian vector field}
$X_H$ and the {\it evolution vector field} $\mc{E}_H$ attached to the
function $H$ are defined as  solutions of the  equations:
\begin{subequations}\label{EE}
  \begin{align}
    \flat(\grad H)  & = \dd H,\label{EE1}\\
    X_H &= \grad H -R(H)R,\label{EE2}\\
    \mc{E}_H & = X_H+ R\label{EE3}.  \end{align}
 \end{subequations}

{\bf TACS} 

In \cite{Albert,can}: {\it transitive almost contact structure} - is  an
almost contact manifold $(M,\theta,\Omega) $ defined at  {\bf ACOS} with \[\dd
  \Omega=0,\] and around every point of $M$
there is a   neighbourhood  where there are local Darboux
coordinates $(\kappa, q^1,\dots,q^n,p_1,\dots, p_n)$ such that
\begin{equation}\label{tac1}
 \theta   =\dd \kappa +\epsilon p_i\dd q^i,\quad
           i=1,\dots,n,~~\epsilon\in \R.
 \end{equation}
To a function $f\in C^{\infty}(M)$ it is associated the Hamiltonian vector field $X_f$
defined by \cite[(3)]{Albert} as solution of the
equations 
\begin{subequations}
\begin{align*}
X_f\lrcorner\theta & =\epsilon f,\\
X_f\lrcorner \Omega & = \dd f-(R f)\theta.
\end{align*}
\end{subequations}

{\bf CH}  

In \cite{MLEO,MLEO1}: {\it contact Hamiltonian system}  $(M,\eta)$--
is defined
as the almost cosymplectic manifold $(M,\eta, \dd \eta)$
and \eqref{thtOM} verified 
\begin{equation}\label{thtOM1}
  \eta\wedge\dd \eta^n \neq 0.
\end{equation}

Apparently the denomination {\it  contact structure} for a manifold
$(M, \eta)$ verifying the condition \eqref{thtOM1} was used firstly by
Gray \cite{GRAY}. 

$\theta$ defined in \eqref{tac1}
with $\epsilon=-1$  corresponds to Darboux coordinates and  in accord
with \cite[Theorem,  page 1]{BL} $\eta$ can be taken 
\begin{equation}\label{tace}
\eta =\dd \kappa - p_i\dd q^i,\quad \dd \eta = \dd q^i\wedge
\dd p_i.
\end{equation}

The conditions \eqref{reeb} defining the  Reeb vector became
\begin{equation}\label{reeb1}
  R\lrcorner \dd \eta =0,\quad R\lrcorner\eta =1, 
\end{equation}
where
\begin{equation}\label{reeb3}
  R=\frac{\pa }{\pa \kappa}.\end{equation}

The {\it Hamiltonian vector field} $X_H$ attached to the real function $H$
is defined by \cite[(2.11)]{MLEO}
\begin{equation}\label{HAM}
\flat(X_H)=\dd H-(R\lrcorner H+H)\eta,
\end{equation}
where the Reeb vector $R$ is defined in \eqref{reeb3} and the
application $\flat$ is defined in \eqref{bemol}.

The vector field $X_H$ in Darboux coordinates reads \cite[(2.12)]{MLEO} 
\begin{equation}\label{517}
  X_H=\frac{\pa H}{\pa p_i}\frac{\pa }{\pa q^i}-
  (\frac{\pa H}{\pa q^i}+\color{green}p_i\frac{\pa H}{\pa \kappa})\color{black}\frac{\pa}{\pa p_i}
  +\color{green}(p_i\frac{\pa H}{\pa p_i}-H)\frac{\pa }{\pa \kappa}\color{black}.
\end{equation}
We determine the vector field $\grad H$ solution of equation   \eqref{EE1}
\begin{equation}\label{GRAD2}
  \grad H=\frac{\pa H}{\pa p_i}\frac{\pa}{\pa q^i}-(\frac{\pa H}{\pa
    q^i}+\color{green}p_i\frac{\pa H}{\pa \kappa}\color{black})\frac{\pa }{\pa p_i}+\color{green}(\frac{\pa
    H}{\pa \kappa}+p_i\frac{\pa H}{\pa p_i})R\color{black},\end{equation}
but instead of \eqref{EE2}, we have 
\begin{equation}\label{EEE2}
X_H = \grad H -(H+R(H))R.
 \end{equation} If in equations \eqref{517} and \eqref{GRAD2} we neglect
 the ``green parts'',  we get $X_H$ \eqref{HXHH} on the symplectic
 manifold $(M,\Omega)$.

 In order to recall the notion of {\it Jacobi bracket}
 \cite{KIR,Lich,VAIS},  we follow \cite[\S 4, Contact manifolds as
 Jacobi structures]{MLEO}.

 The isomorphism \eqref{bemol} becomes
 \[
   X\rightarrow X^{\flat}=X\lrcorner \dd\eta +(X\lrcorner\eta)\eta.
\]
 If \begin{equation}\label{519}
   X=A_i\frac{\pa}{\pa q^i}+B_i\frac{\pa}{\pa p^i}+C\frac{\pa}{\pa
     \kappa},\quad\text{then} \quad X^{\flat}=\alpha_i\dd q^i+\beta_i\dd p_i+
   \gamma \dd \kappa,\quad i=1,\dots,n,
 \end{equation}
 where
 \begin{subequations}\label{bemfl}
   \begin{align} & \flat:\quad  \alpha_i=-B_i+p_i(p_jA_j-C),\quad
                   \beta_i=A_i, \quad \gamma=C-p_jA_j,\\
 & \sharp=\flat^{-1}:\quad  A_i=\beta_i,\quad B_i=-\alpha_i-\gamma p_i,\quad C=p_i\beta_i+\gamma.
   \end{align}
 \end{subequations}
 
 We have to calculate  the bracket   $\{f,g\}_J$ \cite[p. 11]{MLEO}
 \begin{equation}\label{JBB}
   \{f,g\}_J=-\dd \eta(\sharp \dd f,\sharp \dd g)-R(g)f+gR(f).\end{equation}
 With \eqref{GRAD2}, \eqref{519}, \eqref{bemfl}, we find
 \begin{Remark}\label{RR3}
   The
 Jacobi bracket \eqref{JBB} has the expression
\begin{equation}\label{JACP}
  \{f,g\}_J=\{g,f\}_P+\frac{\pa f}{\pa \kappa}g_e-\frac{\pa g}{\pa
    \kappa}f_e,
\end{equation}
  where  $f_e$    denotes  Euler's operator \cite{ARN}
\begin{equation}\label{FE}
  f_e:=f-p_i\frac{\pa f}{\pa p_i}.
\end{equation}

\end{Remark}

In \cite[Chapter V, Contact manifolds]{paulM} the conventions 
expressed in \eqref{thtOM1}, \dots, \eqref{reeb3} are used.

Instead of
\eqref{GEN} it is used \cite[Chapter V, Proposition 6.13 page 293]{paulM} 
\begin{equation}\label{GEN1}
  \eta^{\flat}:X\rightarrow -X\lrcorner\dd \eta, 
\end{equation}
which associates to vector fields semi-basic differential forms
\cite[pages 56, 68]{paulM}. Also
it is used the notation \[
  \eta^{\sharp}(\gamma)=^{\sharp}
  \gamma, \quad \gamma\in \got{D}_1.
  \]
Note that in \cite[Lemma p 44]{paul}
instead of \eqref{GEN1}
it is used 
\[
 \eta^{\flat}:X\rightarrow  X\lrcorner\dd \eta. 
\]
According to \cite[Proposition 6.11 p 292]{paulM}:  {\it We have
 decomposition of the tangent space  $TM$ into the direct sum}  
\[
  TM=\text{Ker}\dd \eta \oplus \text{Ker}~ \eta,
\] 
of the vertical bundle of rank 1 generated by the Reeb vector \eqref{reeb3} and
the horizontal bundle of rank $2n$
\[
  X=(X\lrcorner\eta)R +(X-(X\lrcorner\eta)R).
\]
According to \cite[Proposition 13.1 p 318]{paulM}:
{\it The vector field $X$ on the contact manifold $(M,\eta)$ is an
infinitesimal contact automorphism if and only if there exists a
differentiable function $\rho$ such that} \[\mc{L}_X\eta=\rho \eta.\]

According to  \cite[Theorem
13.3 p 319]{paulM}: {\it The choice of a contact form $\eta$ on the
strictly contact manifold} $(M,\eta)$ {\it defines an isomorphism}
$\Phi:~
L\rightarrow L'$
\[\Phi(X)=X\lrcorner \eta,\quad X_f=\Phi^{-1}(f)=f R+\eta^{\sharp}(\dd
  f-(R\lrcorner \dd f)\eta),\quad \rho=R\lrcorner\dd f,\]
where $L$ ($L'$) is the space of the infinitesimal automorphisms of
$\eta$ (respectively differentiable real-valued functions on $M$).

According to \cite[Proposition 14.2 p 325]{paulM}: {\it The Lie algebra structure of $L'$ is defined by the bracket}
\begin{subequations}\label{PAR}
  \begin{align}
    \{f,g\} & =[X_f,X_g]\lrcorner\eta,\\
    ~& = \dd \eta (X_f,X_g)+f (R\lrcorner\dd g)-g (R\lrcorner \dd
       f),\\
    ~& =\{f,g\}_P+f_e\frac{\pa g}{\pa \kappa}-g_e\frac{\pa f}{\pa
       \kappa}.
       \end{align}
     \end{subequations}
   
\begin{Remark}\label{RR4}
    Note that the expresion of the Jacobi bracket
    \eqref{PAR} is minus the expression
    \eqref{JACP} \[\{f,g\}=-\{f,g\}_J,\] and
  coincides with the expresion \cite[ \{f,g\} on p 636]{Albert} for
  $\epsilon =-1$.
\end{Remark}
 
{\bf SAC} 

In \cite{SB19}: following Sasaki \cite{sas} and \cite[Definition 6.2.5]{boga}, we used the 
\begin{deff}\label{D9}
 The manifold  $M_{2n+1}$ has a  {\it (strict)  almost contact structure}
$(\Phi,\xi,\eta)$  if there exists a
 $(1,1)$-tensor field $\Phi$,  the contravariant vector  field $\xi\in\mc{D}^1$
({\it Reeb vector field}, or {\it characteristic vector field}),
and  the       covariant vector field $\eta\in\got{D}_1$ 
 verifying  the relations 
\begin{equation}\label{ACM}
\eta\lrcorner \xi=1, \quad  \Phi^2X  =-X+\eta(X)\xi .
\end{equation}
\end{deff}

{\bf C}

Following \cite{boy},    in \cite{SB19} we used  the definition   of a
{\it 
contact structure}  ($M_{2n+1},\eta)$ when  $\eta\in \mc{D}_1$ satisfies
\eqref{thtOM1}.

The  {\it 
contact structure} can be given by {\it a codimension one subbundle} $\mc{D}$
{\it of the tangent bundle} $TM$ {\it which is as far from being integrable as
possible}, and $\mc{D}:=\text{Ker}(\eta)$.

{\bf ACM}

Sasaki has proved \cite[Theorem 1.1]{sas} and \cite[(5.16)]{sh}:
\begin{Proposition}\label{THM11}
For an almost contact structure $(\Phi,\xi,\eta)$, the following relations hold
\begin{equation}\label{679}
\Phi \xi=0,\quad 
\eta \Phi =0,\quad 
\text{{\emph{Rank~}}}  (\Phi^i_j)  = 2n, \quad  \xi, \eta^t \in
M(n,1,\R). 
\end{equation}
There exists a
positive Riemannian metric $g$ such that
\begin{equation}\label{680}
 \eta=g \xi, \quad \Phi^tg\Phi=g-\eta^t\otimes\eta. 
\end{equation}
\end{Proposition}
If we put 
\begin{equation}\label{gphi}
  \hat{\Phi}:=g\Phi, 
\end{equation}
the two-form $\hat{\Phi}$ is antisymmetric.

Sasaki has proved \cite[Theorem 3.1]{sas}, \cite{sh},  see also \cite[(9.12 ), Theorem 14]{SB19}:
\begin{Proposition}\label{PP19}
  Let $(M,\eta)$ be a contact manifold. Then we can find an almost
  contact metric  structure $(M,\Phi,\xi,\eta,g)$ such that \eqref{gphi} is satisfied,
  \begin{equation}\label{ECSA}
      \dd \eta (X,Y)=g(X,\Phi(Y)),
      \end{equation}
  and 
  \begin{equation}\label{ECSA2}
    \dd \eta\!\!=\!\!\frac{1}{2}\sum_{i,j=1}^{2n+1}\!\hat{\Phi}_{ij}\dd x^i\!\wedge\!\dd x^j\!=\!\sum_{1\leq i <j
      \leq 2n+1}\!\hat{\Phi}_{ij}\dd x^i\!\wedge\!\dd x^j ,\quad 
  \hat{\Phi}_{ij}\!\!=\!\!
  \pa_i\eta_j\!-\!\pa_j\eta_i.\end{equation}
\end{Proposition}

{\bf N=SAS}

Following \cite[p. 47]{BL}, let us introduce
\begin{deff}Let $h$ be a $(1,1)$-tensor field. Then the
 {\it  Nijenhuis torsion} $[h,h]$ of $h$ is the tensor field of type $(1,2)$
  given by
 \[   [h,h](X,Y)=h^2[X,Y]+[hX,hY]-h[hX,Y]-h[X,hY],\quad
   X,Y\in\got{D}^1.
 \]
\end{deff}

Let us define the $(1,2)$-tensor
\begin{equation}\label{NN1}N^1:=[\Phi,\Phi]+2\dd \eta\otimes\xi.\end{equation}
According to \cite[Theorem 6.5.9]{boga}
 \begin{Proposition}\label{THM15}
 An almost contact structure $(\xi,\eta,\Phi)$ on $M$
   is normal if and only if $N^1=0$. 
 \end{Proposition}
 A normal
contact metric  structure $(M,\xi,\eta,\Phi,g)$ is called a
{\it Sasakian} structure, see  more details \cite[Definitions 6.4.7, 6.5.7, 6.5.13]{boga}. 

Under the conditions \eqref{ACM}, \eqref{680}, but instead of
\eqref{ECSA} and \eqref{NN1} with 
\[
  \dd  \eta (X,Y)=g(\Phi(X), Y),
  \]
  \begin{equation}\label{NN11}N^1:=[\Phi,\Phi]+\dd \eta\otimes\xi,\end{equation}
in  local coordinates $(\kappa, z^1,\dots,z^n)$ on a small neighbourhood
$ \R \times\C^k$ of a Sasaki manifold $(M_{2n+1},\xi,\eta, \Phi,g)$
according with \cite[Theorem 1]{god}, see also \cite{VIS},  we
have 
\begin{subequations}\label{SPOT}
  \begin{align}
    \xi & =\frac{\pa}{\pa \kappa},\\
    \eta & = \dd \kappa +\ii \sum_{j=1}^n(K_{j}\dd z^j-K_{\bar{j}}\dd
           \bar{z}_j),\\
    \dd\eta &= -2\ii \sum_{j,\bar{k}=1}^{n}K_{j\bar{k}}\dd
              z^j\wedge \dd \bar{z}^k,\\
     g & =  \eta\otimes\eta +2\sum_{j,\bar{k}=1}^nK_{j\bar{k}}\dd
            z^j\dd \bar{z}^k,\label{NO4}\\
    \Phi & = -\ii\sum_{j=1}^n(\pa_j-K_j\pa_{\kappa})\otimes \dd z^j+
           \ii\sum_{j=1}^n(\pa_{\bar{j}}-K_{\bar{j}}\pa_{\kappa})\otimes \dd \bar{z}^j,
\end{align}
\end{subequations}
where the {\it Sasaki potential} $K$  does not depend on $\kappa$.

However, {\it using  \eqref{NN1} and \eqref{ECSA}, equation \eqref{NO4}
should be replaced with}
\begin{equation}
  \label{NO44}
  g  =  \eta\otimes\eta-4\sum_{j,\bar{k}=1}^nK_{j\bar{k}}\dd
  z^j\dd \bar{z}^k.
\end{equation}
 In the holonomic cobasis $(\dd y^{\mu})=(\dd \kappa, \dd z^j,\dd
\bar{z}^j)$
the covariant components of the Sasakian metric are \cite[(15)]{god}
\begin{equation}\label{ggold}
  g_{\mu\nu}=\left(\begin{array}{ccc}
 1 &\ii K_j&-\ii K_{\bar{j}}\\
                     \ii K_i& -K_iK_j&
K_{i\bar{j}}+K_iK_{\bar{j}}\\
      -\ii K_{\bar{i}}&
      K_{\bar{i}j}+K_{\bar{i}}K_j&-K_{\bar{i}}K_{\bar{j}}
    \end{array}\right) .
\end{equation}

In the conventions  used in \cite[\S~9.3]{SB19}
for the Heisenberg group $\mr{H}_1$, we make the 
\begin{Remark}\label{RH}
The Sasaki potential for the Heisenberg  group $(\mr{H}_1,\xi,\eta,\Phi,
g_{\mr{H}_1})$ has the
expression
\begin{equation}\label{SP}K_{\mr{H}_1}(z,\bar{z})=\frac{1}{8}(z-\bar{z})^2=-\frac{y^2}{2},\quad
\C\ni  z=x+\ii y, ~x,y\in \R,\end{equation}
corresponding to the left invariant metric
\begin{equation}\label{inm}g_{\mr{H}_1}=\dd x^2+\dd y^2 +\eta^2,\quad \eta=\dd \kappa-y\dd x,\end{equation}
and  \[\Phi= \left(\begin{array}{ccc}0&1&0\\-1& 0&0\\0&y&0\end{array}\right).\]
The \Ka~potential for the manifold $\mr{H}_1/\R\approx \C$
corresponding  to the scalar product of two coherent state vectors
$e_z$ is 
\begin{equation}\label{PH1}
  \mc{K}_{\mr{H}_1} (z,\bar{z}):=(e_z,e_z)=\exp(z\bar{z}).
\end{equation} 
\end{Remark}
\begin{proof}

  We recall   also \cite[Remark 18, Propositions 25, 26] {SB19}.

  We apply \eqref{ggold} for the Heisenberg group  $\mr{H}_1$ with the left
  invariant metric \eqref{inm} \cite[Section 9.3]{SB19}.

  We apply formulas \eqref{SPOT}  with \eqref{NO4}
  replaced with \eqref{NO44} to the Sasaki potential \eqref{SP}.

   Formula \eqref{PH1} is extracted from \cite[(7.75)]{SB19}.
 \end{proof}
{\bf ACOK} 

  In  \cite{Capp}: the notion  of almost cosymplectic manifold 
is adopted as in {\bf ACOS}, the definition of cosymplectic manifold is adopted as in
{\bf COS}, the notion of almost contact structure is introduced as in 
Definition \ref{D9}.

{\it An almost co\Ka~ manifold} -  is an almost contact metric manifold
 $(M,\Phi,\xi,\eta,g)$ such that the two-form $\hat{\Phi}$ \eqref{gphi}
 and $\eta\in\mc{D}_1$ are both closed. If in addition the almost contact structure is normal, $M$
 is called a {\it co\Ka~ manifold}.

According to \cite[Theorem 3.9]{Capp}
\begin{Proposition}
 The almost contact manifold   $(M,\Phi,\xi,\eta,g)$ is
 co\Ka~ if and only if $\nabla^g \Phi=0$ or equivalently $\nabla^g \hat{\Phi}=0$.
\end{Proposition}
 The same definition of co-symplectic/co-\Ka~ manifold
used by \cite{Capp} is adopted in the review paper  \cite{LI}.

Apparently the denomination of  co-\Ka~ manifold was used firstly in
\cite{BEJS}.

In  \cite{BL2}: Blair  used the term cosymplectic manifold denoting  a cosymplectic
manifold in the sense of Libermann \cite{paul} endowed with a compatible almost
contact metric structure as in  {\bf N=SAS}  satisfying the normality condition. i.e. the
co\Ka~ manifold of \cite{Capp}. 

\subsection*{Acknowledgements}
This research was conducted in the framework of the
ANCS project program PN 19 06 01 01/2019. I am grateful to 
Mihai ~Visinescu for references on cosymplectic structures and the remarks
on the  drafts of the paper. I would like to thank Mrs. Popescu Elena
from IMAR for facilitating my access to several  references.

\end{document}